\documentclass[a4paper,10pt]{amsart}
\usepackage[utf8]{inputenc}
 \usepackage{setspace}
\usepackage[cmtip,arrow]{xy}
\usepackage{lscape}
\usepackage{latexsym}
\usepackage[activeacute,english]{babel}
\usepackage{graphicx}
\usepackage{amsmath}
\usepackage{amsfonts}
\usepackage{amsthm}
\usepackage{amssymb}
\usepackage{amscd}
\usepackage{mathrsfs,bbm}
\usepackage{tikz}
\usepackage{pb-diagram,pb-xy,yfonts}
\xyoption{all}
\usepackage{hyperref}
\hypersetup{colorlinks=true,linkcolor=blue,citecolor=red,linktocpage=true}
\usepackage{url}
\usepackage{graphics}

\usepackage{relsize}

\newtheorem{theorem}{Theorem}[section]

\newtheorem{proposition}[theorem]{Proposition}
\newtheorem{corollary}[theorem]{Corollary}
\newtheorem{definition}[theorem]{Definition}
\newtheorem{example}[theorem]{Example}
\newtheorem{remark}[theorem]{Remark}

\newcommand{\C}{\mathcal{C}}

\newcommand{\Mod}{\mathrm{Mod}}

\begin{document}
\title{Gabriel Localization in functor categories}

\author[M. Ortiz., M. L. S. Sandoval-Miranda, V. Santiago]{Mart\'in Ortiz-Morales,\\ Martha Lizbeth Shaid Sandoval-Miranda,\\ Valente Santiago-Vargas}

\thanks{The authors thanks project: Apoyo a la Incorporaci\'on de Nuevos PTC-2019 PRODEP; Grant ID: F-PROMEP-39/Rev-04
   SEP-23-005 }
\date{\today}
\dedicatory{}
\maketitle

\begin{abstract}
P.  Gabriel showed that for a  unital ring $R$, there exists a bijective correspondece between the set of Gabriel filters of $R$ and the set of Giraud subcategories of $\mathrm{Mod}(R)$ (see \cite[Lemme 1]{Gabriel1} on page 412). In this paper we prove an analogous of Gabriel's result: for an small preadditive category $\mathcal{C}$, there exists a bijective correspondence between the Gabriel filters of $\mathcal{C}$ and Giraud subcategories of $\mathrm{Mod}(\mathcal{C})$.
\end{abstract}

\section{Introduction}
The idea that preadditive categories are rings with several objects was developed convincingly by Barry Mitchell (see \cite{Mitchelring})  who showed that a substantial amount of noncommutative ring theory is still true in this generality. Here we would like to emphazise that sometimes  clarity in concepts, statements, and proofs are  gained by dealing with additive categories, and that familiar theorems for rings come out of the natural development of category theory. For instance, the notions of radical of a preadditive category,  perfect  and semisimple rings, global dimensions, among other topics,  have been amply  studied in the context of rings with several objects.\\

In 1962, P. Gabriel introduced in \cite{Gabriel1} the concept of localization in the setting of abelian categories, and he proposed the now so named Gabriel filter on $R$, where $R$ denotes a unital ring and $\mathrm{Mod}(R)$ the category of its left unital $R$-modules, in order to study localization in rings and modules, (see  also \cite{Toma} and \cite[Chapter VI.5]{Stentrom}). Moreover, P. Gabriel showed that there is a bijective correspondence between the set of Gabriel filters of $R$ and the set of class of isomorphisms of Giraud subcategories of $\mathrm{Mod}(R)$ (see \cite{Gabriel1}). Recall that a subcategory $\mathcal{X}$ of $\mathrm{Mod}(R)$ is a Giraud subcategory if the inclusion functor has a left adjoint which is left exact. \\

Through the years, Gabriel filters have been studied by several authors in different contexts. For instance, in \cite{Lidia-TTF}, L. Angeleri H\"ugel and S. Bazzoni studied  Gabriel filters in Grothendieck categories with a generator. Notice that their definition is a little bit different from the one we use through out this paper.\\

In 2015, S. D\'iaz-Alvarado and M. Ort\'iz Morales   introduced in \cite{OrtizDiaz} the notion of Gabriel filter for a preadditive category $\mathcal{C}$ and they proved that there is a bijective correspondence between Gabriel filters of $\mathcal{C}$ and torsion hereditary classes of $\mathrm{Mod}(\mathcal{C})$. Recently, in \cite{Manolo}, C. Parra, M. Saorin and S. Virili have studied  torsion pairs in categories of modules over a preadditive category, abelian recollements by functor categories, and centrally splitting TTFs. 

In their investigation, these authors have given similar definitions and results related to S. D\'iaz-Alvarado and M. Ort\'iz Morales  work.\\

Following Mitchel's philosophy and the definition of Gabriel filter given in the paper \cite{OrtizDiaz}, our  aim in this paper is to study the analogous of Gabriel result  for the contexts of ring with several objects. One of the main results in this work is:\\

{\bfseries Theorem \ref{Biyeccionbuena}} {\it There exists a bijective correspondence between Gabriel filters on $\mathcal{C}$ and the class of isomorphisms of Giraud subcategories of $\mathrm{Mod}(\mathcal{C})$, where $\mathrm{Mod}(\mathcal{C})$ denotes the \textit{category of additive covariant functors} from $\mathcal{C}$ to  the category of abelian groups $ \mathbf{Ab}$.}\\

Our final result is related  with localization by Serre subcategories developed by P. Gabriel. The notion of quotient and localization of abelian categories by dense subcategories (i.e., Serre classes) was introduced by P. Gabriel in his famous  Doctoral thesis ``Des cat\'egories ab\'elienne'' \cite{Gabriel1}, and it plays an important role in ring theory. This notion achieves some goal as quotients in other area of mathematics.  In particular, in this paper we proved that there is an equivalence of categories $\mathrm{Mod}(\mathcal{C})/\mathcal{T}\simeq \mathrm{Mod}(\mathcal{C},\mathcal{F})$ where  $\mathrm{Mod}(\mathcal{C},\mathcal{F})$   is certain Giraud subcategory associated to a Gabriel filter $\mathcal{F}$, and $\mathcal{T}$ is a hereditary torsion class (see \ref{qutirn}).\\

It is worth to mention that tecniques of localization on lattical-contexts has been also studied, for instance, by T. Albu and P. F. Smith, see \cite{Albu_1996}, \cite{Albu_1997}, and\cite{Toma}, and Harold Simmons in a series of unpublished papers.
\\

This article is organized as follows. In Section \ref{sec2}, we recall  definitions needed in the work, and then we collect for later use a variety of results from various backgrounds. In Section \ref{sec3}, we define a prelocalization functor $\mathbb{L}$ and we prove several technical results related to the functor $\mathbb{L}$ that will be needed to define Gabriel localization on $\mathrm{Mod}(\mathcal{C})$. In Section \ref{sec4}, we define the Gabriel localization functor $\mathbb{G}$ and we proved our main result, theorem \ref{Biyeccionbuena}. Finally, in \ref{example} we give an example of a Gabriel filter in the category of representations of an infinite quiver.

\section{Preliminaries}\label{sec2}
We recall that a category $\C$ together with an abelian group structure on each of the sets of morphisms $\C(C_{1},C_{2})$ is called  \textbf{preadditive category}  provided all the composition maps
$\C(C,C')\times \C(C',C'')\longrightarrow \C(C,C'')$
in $ \C $ are bilinear maps of abelian groups. A covariant functor $ F:\C_{1}\longrightarrow \C_{2} $ between  preadditive categories $ \C_{1} $ and $ \C_{2} $ is said to be \textbf{additive} if for each pair of objects $ C $ and $ C' $ in $ \C_{1}$, the map $ F:\C_{1}(C,C')\longrightarrow \C_{2}(F(C),F(C')) $ is a morphism of abelian groups. Let $\mathcal C$ and $\mathcal D$  be preadditive categories and $\mathbf{Ab}$ the category of abelian groups.

\subsection{The category $\mathrm{Mod}(\mathcal{C})$}
Throughout this section $\mathcal{C}$ will be an arbitrary skeletally small preadditive category, and $\mathrm{Mod}(\mathcal{C})$ will denote the \textit{category of additive covariant functors} from $\mathcal{C}$ to  the category of abelian groups $ \mathbf{Ab}$, called the category of $\mathcal{C}$-modules. This category has as objects  the functors from $\mathcal C$ to $\mathbf{Ab}$, and  a morphism $ f:M_{1}\longrightarrow M_{2} $ of $ \C $-modules is a natural transformation,  that is, the set of morphisms $\mathrm{Hom}_\mathcal C(M_1,M_2)$ from $M_1$ to $M_2$  is given by $\mathrm{Nat} (M_{1}, M_{2} )$.  We sometimes we will write for short, $\mathcal{C}(-,?)$
instead of $\mathrm{Hom}_{\mathcal{C}}(-,?)$ and when it is clear from the context we will use just $(-,?).$
\\
	
We now recall some properties of the category $ \Mod(\C).$ The category $\Mod(\C) $ is a Grothendieck category with the following properties:
\begin{enumerate}
\item A sequence
\[
\begin{diagram}
\node{M_{1}}\arrow{e,t}{f}
 \node{M_{2}}\arrow{e,t}{g}
  \node{M_{3}}
\end{diagram}
\]
is exact in $ \Mod(\C) $ if and only if
\[
\begin{diagram}
\node{M_{1}(C)}\arrow{e,t}{f_{C}}
 \node{M_{2}(C)}\arrow{e,t}{g_{C}}
  \node{M_{3}(C)}
\end{diagram}
\]
is an exact sequence of abelian groups for each $ C$ in $\C $.

\item Let $ \lbrace M_{i}\rbrace_{i\in I} $ be a family of $ \C $-modules indexed by the set $ I $.
The $ \C $-module $ \underset{\i\in I}\amalg M_{i}$ defined by $  (\underset{i\in I}\amalg M_{i})\ (C)=\underset{i\in I}\amalg \ M_{i}(C)$ for all $ C $ in $ \C $, is a direct sum for the family $ \lbrace M_{i}\rbrace_{i\in I} $ in $ \Mod(\C) $, where $\underset{i\in I} \amalg M_{i}(C)  $ is the direct sum in $ \mathbf{Ab} $ of the family of abelian groups $ \lbrace M_{i}(C)\rbrace_{i\in I} $.  The $ \C $-module $ \underset{\i\in I}\prod M_{i}$ defined by $(\underset{i\in I}\prod M_{i})\ (C)=\underset{i\in I}\prod M_{i}(C)   $ for all $C$ in $\C$, is a product for the family $ \lbrace M_{i}\rbrace_{i\in I} $ in $\Mod(\C)$, where $ \underset{i\in I}\prod M_{i}(C)  $ is the product in $\mathbf{Ab}$.

\item  (Yoneda Lemma) For each $C$ in $\C $, the $\C$-module $(C,-)$ given by $(C,-)(X)=\C(C,X)$ for each $X$ in $\C$, has the property that for each $\C$-module $M$, the map $\left( (C,-),M\right)\longrightarrow M(C)$ given by $f\mapsto f_{C}(1_{C})$ for each $\C$-morphism $f:(C,-)\longrightarrow M$ is an isomorphism of abelian groups. We will often consider this isomorphism an identification.
Hence
\begin{enumerate}
\item The functor $ P:\C\longrightarrow \Mod(\C) $ given by $ P(C)=(C,-) $ is fully faithful.
\item For each family $\lbrace  C_{i}\rbrace _{i\in I}$ of objects in $ \C $, the $ \C $-module $ \underset{i\in I}\amalg P(C_{i}) $ is a projective $ \C $-module.
\item Given a $ \C $-module $ M $, there is a family $ \lbrace C_{i}\rbrace_{i\in I} $ of objects in $ \C $ such that there is an epimorphism $ \underset{i\in I}\amalg P(C_{i})\longrightarrow M\longrightarrow 0 $.
\end{enumerate}
\end{enumerate}

The reader can see \cite{MitBook} and \cite{AusM1} for more details of all these facts.

\subsection{Linear filters of $\mathcal{C}$}

In \cite{OrtizDiaz} were studied the notion of linear filters in preadditive categories, and there were given versions of classical definitions of (Gabriel) filters, torsion theories and annihilators of ideals in the category $\mathrm{Mod}(\mathcal{\C})$ of additive functors from $\mathcal{C}$ to $\bf{Ab}.$ 
Generalizations of classical results were obtained, such as the theorem explained by Gabriel that establishes a bijective correspondence between hereditary torsion theories and linear filters.

Next, we recall some basic notions introduced in \cite{OrtizDiaz}:

\begin{definition}
An additive subfunctor $I(C,-)$ of the functor $\mathrm{Hom}_{\mathcal{C}}(C,-)$  is called a $\textbf{left ideal}$ of $\mathrm{Hom}_{\mathcal{C}}(C,-)$.
\end{definition}
We will write sometimes $I$ instead of $I(C,-)$ when it is clear from the context  that $I$ is a subfunctor of $\mathrm{Hom}_{\mathcal{C}}(C,-)$.
With the above definition, we recall that a $\textbf{left ideal}$ $\textbf{in an additive category}$ $\mathcal{C}$ is a collection of left ideals 
$$\{I(C,-)\subseteq \mathrm{Hom}_{\mathcal{C}}(C,-)\mid C\in \mathcal{C}\}.$$
Similarly we can define a right ideal of $\mathcal{C}$. A $\textbf{two sided ideal of the category}$ $\mathcal{C}$  is an additive subfunctor of the two variable functor $\mathrm{Hom}_{\mathcal{C}}(-,?):\mathcal{C}^{op}\times \mathcal{C}\longrightarrow \mathbf{Ab}$.\\
\begin{definition}\label{anulador}
\begin{enumerate}
\item [(a)] Let $N\in \mathrm{Mod}(\mathcal{C})$ be and $K$ a submodule  of $N$.  Consider $C\in \mathcal{C}$ and $x\in N(C)$ we define the following $\mathcal{C}$-module
$$\big(K(-):x\big):\mathcal{C}\longrightarrow \mathbf{Ab},$$
as follows: for $C'\in \mathcal{C}$ we set
$$\big(K(-):x\big)(C'):=\{f\in\mathrm{Hom}_{\mathcal{C}}(C,C')\mid N(f)(x)\in K(C')\}.$$

\item [(b)] Let $0$ be the zero module in $\mathrm{Mod}(\mathcal{C})$, since $0$ is a submodule of $N$.  We define the ideal $\textbf{annihilator of}$ $\mathbf{x}\in N(C)$ denoted as $\mathrm{Ann}(x,-)$ as follows:
$$\mathrm{Ann}(x,-):=\big( 0(-):x\big).$$
That is, $\mathrm{Ann}(x,-)(C'):=\{f\in\mathrm{Hom}_{\mathcal{C}}(C,C')\mid N(f)(x)=0\}$ for $C'\in \mathcal{C}$. 
\end{enumerate}
\end{definition}
It is easy to see that $\big(K(-):x\big)$ is a left ideal of $\mathrm{Hom}_{\mathcal{C}}(C,-)$ and that for each $x\in N(C)$, we have the left ideal $\mathrm{Ann}(x,-)$ of $\mathrm{Hom}_{\mathcal{C}}(C,-)$.

\begin{remark}\label{idealpull}
\begin{enumerate}
\item [(a)]
Taking $N=\mathrm{Hom}_{\mathcal{C}}(C,-)$ for some $C\in \mathcal{C}$  and $K=I(C,-)$ a left ideal of $N$, for $h\in N(B)=\mathrm{Hom}_{\mathcal{C}}(C,B)$ we have the following $\mathcal{C}$-module
$$\big(I(C,-):h\big):\mathcal{C}\longrightarrow \mathbf{Ab},$$
defined as follows: for $C'\in \mathcal{C}$ we set
$$\big(I(C,-):h\big)(C'):=\{f\in\mathrm{Hom}_{\mathcal{C}}(B,C')\mid f\circ h\in I(C,C')\}.$$ Then 
$\big(I(C,-):h\big)$ is a left ideal of $\mathrm{Hom}_{\mathcal{C}}(B,-).$

\item [(b)] We have that $\big(I(C,-):h\big)$  is given by the following pullback in $\mathrm{Mod}(\mathcal{C})$
$$\xymatrix{\big(I(C,-):h\big)\ar[rr]^{\gamma_{I}^{C}}\ar[d]  & & I(C,-)\ar[d]^{\delta_{I}}\\
\mathrm{Hom}_{\mathcal{C}}(B,-)\ar[rr]^{\mathrm{Hom}_{\mathcal{C}}(h,-)} & & \mathrm{Hom}_{\mathcal{C}}(C,-)}$$
where $\gamma_{I}^{C}:=\mathrm{Hom}_{\mathcal{C}}(h,-)|_{(I(C,-):h)}$ and  $\delta_{I}$ is the inclusion of $I(C,-)$ into $\mathrm{Hom}_{\mathcal{C}}(C,-)$. That is, if $\alpha:=\mathrm{Hom}_{\mathcal{C}}(h,-)$ we have that 
$$\big(I(C,-):h\big)=\alpha^{-1}(I(C,-)).$$
Let us denote by $\delta_{\alpha^{-1}(I)}:\alpha^{-1}(I(C,-))\longrightarrow \mathrm{Hom}_{\mathcal{C}}(B,-)$ the canonical inclusion.
\end{enumerate}
\end{remark}

\begin{definition}\cite[Definition 2.2]{OrtizDiaz}\label{filterdef}
Let $\mathcal{F}_{C}$ be a family of left ideals of $\mathrm{Hom}_{\mathcal{C}}(C,-)$. It is said that $\mathcal{F}_{C}$ is a $\textbf{left filter}$ of $\mathrm{Hom}_{\mathcal{C}}(C,-)$ if the following conditions hold:
\begin{enumerate}
\item [($T_{1}$)] If $I\in \mathcal{F}_{C}$ and $I\subseteq J$ then $J\in \mathcal{F}_{C}$,

\item [($T_{2}$)] If $I,J\in \mathcal{F}_{C}$ then $I\cap J\in \mathcal{F}_{C}$.
\end{enumerate}
A collection $\mathcal{F}:=\{\mathcal{F}_{C}\}_{C\in \mathcal{C}}$ is a $\textbf{left linear filter for the category}$ $\mathcal{C}$ if $\mathcal{F}_{C}$ is a filter for $\mathrm{Hom}_{\mathcal{C}}(C,-)$ for all $C\in \mathcal{C}$ and

\begin{enumerate}
\item [($T_{3}$)] If $I\in \mathcal{F}_{C}$ and $h:C\longrightarrow B$ is a morphism in $\mathcal{C}$ then $\big(I(C,-):h\big)\in \mathcal{F}_{B}$ (see \ref{idealpull}(a)).
\end{enumerate}
A collection $\mathcal{F}=\{\mathcal{F}_{C}\}_{C\in \mathcal{C}}$ is a $\textbf{left Gabriel filter for the category}$ $\mathcal{C}$ if $\mathcal{F}=\{\mathcal{F}_{C}\}_{C\in \mathcal{C}}$ is a linear filter and the following holds:
\begin{enumerate}
\item [($T_{4}$)]  Let $J(C,-) \in \mathcal{F}_{C}$ be and $I(C,-)$ an ideal satisfiyng that for each $B\in \mathcal{C}$ the ideal $\big(I(C,-):h)$ belongs to $\mathcal{F}_{B}$ for all $h\in J(C,B)\subseteq \mathrm{Hom}_{\mathcal{C}}(C,B)$, then $I(C,-)\in \mathcal{F}_{C}$.
\end{enumerate}
\end{definition}

\begin{definition}\cite{OrtizDiaz} \label{filterpretorclas}
\begin{enumerate}
\item [(a)] 
Let $\mathcal{F}=\{\mathcal{F}_{C}\}_{C\in \mathcal{C}}$ be a linear filter in $\mathcal{C}$. We define
$$\mathcal{T}_{\mathcal{F}}:=\left\{M\in \mathrm{Mod}(\mathcal{C})\mid  \text{for each}\,\, C\in \mathcal{C},\,\,\, \mathrm{Ann}(x,-)\in \mathcal{F}_{C}\,\,\forall x\in M(C) \right\}.$$

\item [(b)] Let $\mathcal{T}$ be a hereditary pretorsion class in $\mathrm{Mod}(\mathcal{C})$. We define $\mathcal{F}_{\mathcal{T}}:=\{\mathcal{F}_{C}\}_{C\in \mathcal{C}}$ where 
$$\mathcal{F}_{C}:=\left\{I\subseteq \mathrm{Hom}_{\mathcal{C}}(C,-)\mid \frac{\mathrm{Hom}_{\mathcal{C}}(C,-)}{I}\in \mathcal{T}\right\}.$$
\end{enumerate}
\end{definition}

\begin{definition}\label{Ftorsion}
Let $\mathcal{F}=\{\mathcal{F}_{C}\}_{C\in \mathcal{C}}$ be a left filter in $\mathcal{C}$. We say that $M\in \mathrm{Mod}(\mathcal{C})$ is an $\mathcal{F}$-$\textbf{torsion}$ $\textbf{module}$ if $M\in \mathcal{T}_{\mathcal{F}}$.
\end{definition}

We recall that a class $\mathcal{A}\subseteq \mathrm{Mod}(\mathcal{C})$ is a pretorsion class if it is closed under quotient objects and coproducts. A pretorsion class is called hereditary if it is closed under subobjects. A  hereditary torsion class is a hereditary pretorsion class which is closed under extensions. Then we have the following results.

\begin{theorem}\cite[Theorem 2.5]{OrtizDiaz}\label{biyefilterpretor}
The maps $\mathcal{F}\rightarrow \mathcal{T}_{\mathcal{F}}$, $\mathcal{T}\rightarrow \mathcal{F}_{\mathcal{T}}$ induce a bijection between hereditary pretorsion classes of $\mathrm{Mod}(\mathcal{C})$ and left linear filters on $\mathcal{C}$.
\end{theorem}

\begin{theorem}\cite[Theorem 2.6]{OrtizDiaz}\label{biyeGafilterpretor} The maps $\mathcal{F}\rightarrow \mathcal{T}_{\mathcal{F}}$, $\mathcal{T}\rightarrow \mathcal{F}_{\mathcal{T}}$ induce a bijection 
$$\xymatrix{\{\text{Left Gabriel filters of}\,\,\mathcal{C}\}
\ar@<1ex>[d]^{\Psi}\\
\{\text{Hereditary torsion classes of}\,\,\mathrm{Mod}(\mathcal{C})\}
\ar@<1ex>[u]^{\Psi^{-1}}}$$
\end{theorem}
We recall also that a preradical $t$ of $\mathrm{Mod}(\mathcal{C})$ is just a subfunctor of the identity functor $1_{\mathrm{Mod}(\mathcal{C})}:\mathrm{Mod}(\mathcal{C})\longrightarrow \mathrm{Mod}(\mathcal{C}).$ A preradical $t$ is called radical if $t\Big(\frac{M}{t(M)}\Big)=0$ for all  $M\in \mathrm{Mod}(\mathcal{C})$.\\
Let $\mathcal{T}$ be a pretorsion class in $\mathrm{Mod}(\mathcal{C})$. We can construct a preradical $t_{\mathcal{T}}$ associated to this pretorsion class as follows: For $M\in \mathrm{Mod}(\mathcal{C})$
$$t_{\mathcal{T}}(M)=\sum_{N\in \mathcal{T},N\subseteq M}N,$$
and for $f:M\longrightarrow N$ we have $t_{\mathcal{T}}(f):=f|_{t_{\mathcal{T}}(M)}$.\\
Conversely, let $t$ be a preradical in $\mathrm{Mod}(\mathcal{C})$. We construct the class $\mathcal{T}_{t}:=\{M\in\mathrm{Mod}(\mathcal{C})\mid t(M)=M\}$.
We have the following well known results.
\begin{proposition}\label{biyepretorprerad}
The maps $\mathcal{T}\rightarrow t_{\mathcal{T}}$ and $t\rightarrow \mathcal{T}_{t}$ give a bijective correspondence between left exact preradicals of $\mathrm{Mod}(\mathcal{C})$ and hereditary pretorsion classes of $\mathrm{Mod}(\mathcal{C})$
\end{proposition}
\begin{proof}
See \cite[Corollary 1.8]{Stentrom} on page 138.
\end{proof}

\begin{proposition}\label{biyepretorprerad}
The maps $\mathcal{T}\rightarrow t_{\mathcal{T}}$ and $t\rightarrow \mathcal{T}_{t}$ give a bijective correspondence
$$\xymatrix{\{\text{Hereditary torsion classes of}\,\,\mathrm{Mod}(\mathcal{C})\}
\ar@<1ex>[d]^{\Theta}\\
\{\text{Left exact radicals of}\,\,\mathrm{Mod}(\mathcal{C})\}
\ar@<1ex>[u]^{\Theta^{-1}}}$$
\end{proposition}
\begin{proof}
See \cite[Proposition 3.1]{Stentrom} on page 141.
\end{proof}
Now, consider a left Gabriel filter $\mathcal{F}:=\{\mathcal{F}_{C}\}_{C\in \mathcal{C}}$ in $\mathcal{C}$. By \ref{biyeGafilterpretor} we have the hereditary torsion class
$$\mathcal{T}_{\mathcal{F}}:=\left\{M\in \mathrm{Mod}(\mathcal{C})\mid  \text{for each}\,\, C\in \mathcal{C},\,\,\, \mathrm{Ann}(x,-)\in \mathcal{F}_{C}\,\,\forall x\in M(C) \right\}.$$
By \ref{biyepretorprerad}, we have the corresponding left exact radical $t$, (we use $t$ instead of $t_{\mathcal{T}_{\mathcal{F}}}$ to avoid such a horrible notation) defined as:
$$t(M)=\sum_{N\in \mathcal{T}_{\mathcal{F}},N\subseteq M}N.$$

\begin{remark}\label{remtorsiorad}
Let $\mathcal{F}:=\{\mathcal{F}_{C}\}_{C\in \mathcal{C}}$ be a  left Gabriel filter in $\mathcal{C}$ and $t$ the radical associated to the filter $\mathcal{F}$ via the bijections \ref{biyeGafilterpretor} and \ref{biyepretorprerad}. Then we have that $M$ is an $\mathcal{F}$-torsion  module if and only if $t(M)=M$.
\end{remark}

\section{Prelocalization functor}\label{sec3}

Let  $\mathcal{F}:=\{\mathcal{F}_{C}\}_{C\in \mathcal{C}}$ be a left linear filter on $\mathcal{C}$ as defined in \ref{filterdef}.  Then, we have that for each $C\in \mathcal{C}$  the set of left ideals $\mathcal{F}_{C}$ is a directed set with  the order defined as follows:
$$J\leq I \Longleftrightarrow I\subseteq J. $$ 
We recall that  we will write $I$ instead of $I(C,-)$ when it is clear from the context  that $I$ is a subfunctor of $\mathrm{Hom}_{\mathcal{C}}(C,-)$. Now, let $I\subseteq J$ in $\mathcal{F}_{C}$ and let us denote by $\mu_{I,J}:I\longrightarrow J$ the  canonical inclusion.  For $M\in \mathrm{Mod}(\mathcal{C})$, we have a morphism of abelian groups
$$\lambda_{J,I}:=\mathrm{Hom}_{\mathrm{Mod}(\mathcal{C})}(\mu_{I,J},M):\mathrm{Hom}_{\mathrm{Mod}(\mathcal{C})}\big(J,M\big)\longrightarrow \mathrm{Hom}_{\mathrm{Mod}(\mathcal{C})}\big(I,M\big).$$ Then we have a directed system of abelian groups
$$\left\{\lambda_{J,I}:\mathrm{Hom}_{\mathrm{Mod}(\mathcal{C})}\big(J,M\big)\longrightarrow \mathrm{Hom}_{\mathrm{Mod}(\mathcal{C})}\big(I,M\big)\right\}_{J\leq I}.$$
Thus we can form the abelian group 
$$\varinjlim_{I\in \mathcal{F}_{C}}\mathrm{Hom}_{\mathrm{Mod}(\mathcal{C})}\big(I,M\big).$$ For each $I\in \mathcal{F}_{C}$  we will denote by
$$\varphi_{I,M}^{C}:\mathrm{Hom}_{\mathrm{Mod}(\mathcal{C})}(I,M)\longrightarrow \varinjlim_{I\in \mathcal{F}_{C}}\mathrm{Hom}_{\mathrm{Mod}(\mathcal{C})}\big(I,M\big)$$
the canonical morphisms into the direct limit.
Now we define a functor
$$\mathbb{L}(M):\mathcal{C}\longrightarrow \mathbf{Ab}$$
as follows:
\begin{enumerate}
\item [(i)] $\mathbb{L}(M)(C):=\varinjlim_{I\in \mathcal{F}_{C}}\mathrm{Hom}_{\mathrm{Mod}(\mathcal{C})}\big(I,M\big)$.

\item [(ii)] If $h:C\longrightarrow B$ is a morphism in $\mathcal{C}$ we need to construct
$$\mathbb{L}(M)(h):\mathbb{L}(M)(C)\longrightarrow \mathbb{L}(M)(B).$$
Indeed, we have $\alpha:=\mathrm{Hom}_{\mathcal{C}}(h,-):\mathrm{Hom}_{\mathcal{C}}(B,-)\longrightarrow \mathrm{Hom}_{\mathcal{C}}(C,-)$. By \ref{idealpull}(b) we have that 
$$(I(C,-):h):=\alpha^{-1}(I(C,-)),$$
and we also have the diagram
$$\xymatrix{\alpha^{-1}(I(C,-))\ar[rr]^{\gamma_{C}^{I}}\ar[d]  & & I(C,-)\ar[d]^{\delta_{I}}\\
\mathrm{Hom}_{\mathcal{C}}(B,-)\ar[rr]_{\mathrm{Hom}_{\mathcal{C}}(f,-)} && \mathrm{Hom}_{\mathcal{C}}(C,-)}$$
where $\gamma_{C}^{I}:=\mathrm{Hom}_{\mathcal{C}}(h,-)|_{(I(C,-):h)}$ and $\delta_{I}:I(C,-)\longrightarrow \mathrm{Hom}_{\mathcal{C}}(C,-)$ is the inclusion.\\
Using the universal property of the pullback for $I(C,-)\subseteq J(C,-)\in \mathcal{F}_{C}$  with inclusion $\mu_{I,J}:I(C,-)\longrightarrow J(C,-)$ we have the diagram
$$\xymatrix{\alpha^{-1}(I(C,-))\ar[rrr]^{\gamma_{C}^{I}}\ar[d]_{\mu_{\alpha^{-1}(I),\alpha^{-1}(J)}}  & & & I(C,-)\ar[d]^{\mu_{I,J}}\\
\alpha^{-1}(J(C,-))\ar[rrr]_{\gamma_{C}^{J}} & & & J(C,-)}$$
where $\mu_{\alpha^{-1}(I),\alpha^{-1}(J)}:\alpha^{-1}(I)\longrightarrow \alpha^{-1}(J)$ denotes the inclusion of $\alpha^{-1}(I)$ into $\alpha^{-1}(J)$. By applyin $\mathrm{Hom}_{\mathcal{C}}(-,M)$ to the previous diagram  we have
$$\xymatrix{\mathrm{Hom}_{\mathrm{Mod}(\mathcal{C})}\big(J,M\big)\ar[rrr]^{\mathrm{Hom}_{\mathrm{Mod}(\mathcal{C})}\big(\gamma_{J}^{C},M\big)}\ar[d]^{\lambda_{J,I}} & & &  \mathrm{Hom}_{\mathrm{Mod}(\mathcal{C})}\big(\alpha^{-1}(J),M\big)\ar[d]^{\lambda_{\alpha^{-1}(J),\alpha^{-1}(I)}}\\
\mathrm{Hom}_{\mathrm{Mod}(\mathcal{C})}\big(I,M\big)\ar[rrr]^{\mathrm{Hom}_{\mathrm{Mod}(\mathcal{C})}\big(\gamma_{I}^{C},M\big)} & &  &  \mathrm{Hom}_{\mathrm{Mod}(\mathcal{C})}\big(\alpha^{-1}(I),M\big).}$$
Now, lets consider the canonical morphisms into the direct limits 
$$\varphi_{I,M}^{C}:\mathrm{Hom}_{\mathrm{Mod}(\mathcal{C})}(I,M)\longrightarrow \varinjlim_{I\in \mathcal{F}_{C}}\mathrm{Hom}_{\mathrm{Mod}(\mathcal{C})}\big(I,M\big),$$ 
since $\varphi^{-1}(I)\in \mathcal{F}_{B}$ ($\mathcal{F}=\{\mathcal{F}_{C}\}_{C\in \mathcal{C}}$ satisfies $T3$) we also have  the following canonical morphisms
$$\varphi_{\alpha^{-1}(I),M}^{B}:\mathrm{Hom}_{\mathrm{Mod}(\mathcal{C})}(\alpha^{-1}(I),M)\longrightarrow \varinjlim_{I'\in \mathcal{F}_{B}}\mathrm{Hom}_{\mathrm{Mod}(\mathcal{C})}\big(I',M\big)$$ into the direct limit $\varinjlim_{I'\in \mathcal{F}_{B}}\mathrm{Hom}_{\mathrm{Mod}(\mathcal{C})}\big(I',M\big)$.\\
Then, there exists a unique morphism 
$$\mathbb{L}(M)(h):\mathbb{L}(M)(C)\longrightarrow \mathbb{L}(M)(B)$$ such that the following diagram commutes for all $I\in\mathcal{F}_{C}$
\begin{equation}
\xymatrix{\mathbb{L}(M)(C)\ar[rrr]^{\mathbb{L}(M)(h)} & & & \mathbb{L}(M)(B)\\
\mathrm{Hom}_{\mathrm{Mod}(\mathcal{C})}\big(I,M\big)\ar[rrr]^{\mathrm{Hom}_{\mathrm{Mod}(\mathcal{C})}\big(\gamma_{I}^{C},M\big)}\ar[u]^{\varphi_{I,M}^{C}} & & &  \mathrm{Hom}_{\mathrm{Mod}(\mathcal{C})}\big(\alpha^{-1}(I),M\big)\ar[u]_{\varphi_{\alpha^{-1}(I),M}^{B}}.}
\end{equation}
\end{enumerate}

\begin{definition}\label{funtorL}
We define the functor $\mathbb{L}:\mathrm{Mod}(\mathcal{C})\longrightarrow \mathrm{Mod}(\mathcal{C})$ as follows:
\begin{enumerate}
\item [(a)] For $M\in \mathrm{Mod}(\mathcal{C})$ we define
$\mathbb{L}(M)\in \mathrm{Mod}(\mathcal{C})$ as the functor defined above.

\item [(c)] Let $\eta:M\longrightarrow N$ a natural transformation, then we get  the following commutative diagram for all $J\leq I$ in $\mathcal{F}_{C}$
$$\xymatrix{\mathrm{Hom}_{\mathrm{Mod}(\mathcal{C})}\big(J,M\big)\ar[rr]^{\mathrm{Hom}_{\mathrm{Mod}(\mathcal{C})}\big(J,\eta\big)}\ar[d]_{\mathrm{Hom}_{\mathrm{Mod}(\mathcal{C})}\big(\mu_{I,J},M\big)}& &  \mathrm{Hom}_{\mathrm{Mod}(\mathcal{C})}\big(J,N\big)\ar[d]^{\mathrm{Hom}_{\mathrm{Mod}(\mathcal{C})}\big(\mu_{I,J},N\big)}\\
\mathrm{Hom}_{\mathrm{Mod}(\mathcal{C})}\big(I,M\big)\ar[rr]^{\mathrm{Hom}_{\mathrm{Mod}(\mathcal{C})}\big(I,\eta\big)} & &  \mathrm{Hom}_{\mathrm{Mod}(\mathcal{C})}\big(I,N\big).}$$
Then we have a unique morphism of abelian groups
$$\overline{\eta}_{C}:\varinjlim_{I\in \mathcal{F}_{C}}\mathrm{Hom}_{\mathrm{Mod}(\mathcal{C})}\big(I,M\big)\longrightarrow \varinjlim_{I\in \mathcal{F}_{C}}\mathrm{Hom}_{\mathrm{Mod}(\mathcal{C})}\big(I,N\big)$$ such that the following diagram commutes
\begin{equation}
\xymatrix{\varinjlim_{I\in \mathcal{F}_{C}}\mathrm{Hom}_{\mathrm{Mod}(\mathcal{C})}\big(I,M\big)\ar[rr]^{\overline{\eta}_{C}} & & \varinjlim_{I\in \mathcal{F}_{C}}\mathrm{Hom}_{\mathrm{Mod}(\mathcal{C})}\big(I,N\big)\\
\mathrm{Hom}_{\mathrm{Mod}(\mathcal{C})}\big(I,M\big)\ar[rr]^{\mathrm{Hom}_{\mathrm{Mod}(\mathcal{C})}(I,\eta)}\ar[u]^{\varphi_{I,M}^{C}} & & \mathrm{Hom}_{\mathrm{Mod}(\mathcal{C})}\big(I,N\big)\ar[u]^{\varphi_{I,N}^{C}}.}
\end{equation}
We define $\mathbb{L}(\eta):=\overline{\eta}$, where $\overline{\eta}:=\{\overline{\eta}_{C}\}_{C\in \mathcal{C}}$.
\end{enumerate}
\end{definition}

\begin{proposition}
The functor $\mathbb{L}:\mathrm{Mod}(\mathcal{C})\longrightarrow \mathrm{Mod}(\mathcal{C})$ is left exact.
\end{proposition}
\begin{proof}
This follows from the fact that $\mathrm{Hom}_{\mathrm{Mod}(\mathcal{C})}(I,-)$ is left exact and  $\varinjlim$ is exact, because $\mathbf{Ab}$ is a Grothendieck category.
\end{proof}

Now,  let us consider $C\in \mathcal{C}$ and $I\subseteq J$ in $\mathcal{F}_{\mathcal{C}}$, we have the following commutative diagram
$$\xymatrix{\mathrm{Hom}_{\mathrm{Mod}(\mathcal{C})}\big(\mathrm{Hom}_{\mathcal{C}}(C,-),M\big)\ar[rr]^{\theta_{J,M}^{C}}\ar@{=}[d] & &  \mathrm{Hom}_{\mathrm{Mod}(\mathcal{C})}\big(J,M\big)\ar[d]\\
\mathrm{Hom}_{\mathrm{Mod}(\mathcal{C})}\big(\mathrm{Hom}_{\mathcal{C}}(C,-),M\big)\ar[rr]_{\theta_{I,M}^{C}} & &  \mathrm{Hom}_{\mathrm{Mod}(\mathcal{C})}\big(I,M\big)}$$
where $\theta_{J,M}^{C}:=\mathrm{Hom}_{\mathrm{Mod}(\mathcal{C})}(\delta_{J},M)$ and $\theta_{I,M}^{C}:=\mathrm{Hom}_{\mathrm{Mod}(\mathcal{C})}(\delta_{I},M)$.\\
So, it induces a morphism
$$[\psi_{M}]_{C}:\varinjlim_{I\in \mathcal{F}_{C}}\mathrm{Hom}_{\mathrm{Mod}(\mathcal{C})}\big(\mathrm{Hom}_{\mathcal{C}}(C,-),M\big)\longrightarrow \varinjlim_{I\in \mathcal{F}_{C}}\mathrm{Hom}_{\mathrm{Mod}(\mathcal{C})}\big(I,M\big).$$
(recall that $\varinjlim_{I\in \mathcal{F}_{C}}\mathrm{Hom}_{\mathrm{Mod}(\mathcal{C})}\big(\mathrm{Hom}_{\mathcal{C}}(C,-),M\big)=\mathrm{Hom}_{\mathrm{Mod}(\mathcal{C})}\big(\mathrm{Hom}_{\mathcal{C}}(C,-),M\big)$) 
such that the following diagram commutes for each $I\in \mathcal{F}_{C}$:
$$\xymatrix{\mathrm{Hom}_{\mathrm{Mod}(\mathcal{C})}\big(\mathrm{Hom}_{\mathcal{C}}(C,-),M\big)\ar[r]^{[\psi_{M}]_{C}}\ar@{=}[d] &  \varinjlim_{I\in \mathcal{F}_{C}}\mathrm{Hom}_{\mathrm{Mod}(\mathcal{C})}\big(I,M\big)\\
\mathrm{Hom}_{\mathrm{Mod}(\mathcal{C})}\big(\mathrm{Hom}_{\mathcal{C}}(C,-),M\big)\ar[r]_{\theta_{I,M}^{C}} &  \mathrm{Hom}_{\mathrm{Mod}(\mathcal{C})}\big(I,M\big)\ar[u]_{\varphi_{I,M}^{C}}}$$
Now, consider the Yoneda isomorphism $Y_{C}:M(C)\longrightarrow \mathrm{Hom}_{\mathrm{Mod}(\mathcal{C})}\big(\mathrm{Hom}_{\mathcal{C}}(C,-),M\big)$, then we have the following commutative diagram for each $I\in\mathcal{F}_{C}$
\begin{equation}
\xymatrix{M(C)\ar[r]^(.3){Y_{C}}\ar@{=}[d] & \mathrm{Hom}_{\mathrm{Mod}(\mathcal{C})}\big(\mathrm{Hom}_{\mathcal{C}}(C,-),M\big)\ar[r]^{[\psi_{M}]_{C}}\ar@{=}[d] &  \varinjlim_{I\in \mathcal{F}_{C}}\mathrm{Hom}_{\mathrm{Mod}(\mathcal{C})}\big(I,M\big)\\
M(C)\ar[r]^(.3){Y_{C}} & \mathrm{Hom}_{\mathrm{Mod}(\mathcal{C})}\big(\mathrm{Hom}_{\mathcal{C}}(C,-),M\big)\ar[r]_{\theta_{I,M}^{C}} &  \mathrm{Hom}_{\mathrm{Mod}(\mathcal{C})}\big(I,M\big)\ar[u]_{\varphi_{I,M}^{C}}}
\end{equation}

\begin{definition}\label{defimorfismovar}
Let $M\in \mathrm{Mod}(\mathcal{C})$  be. 
For each $C\in \mathcal{C}$ we define
$[\varphi_{M}]_{C}:M(C)\longrightarrow \mathbb{L}(M)(C)$ as follows:
$$[\varphi_{M}]_{C}:=[\psi_{M}]_{C}\circ Y_{C}$$
\end{definition}
Thus, we have the following result.

\begin{proposition}\label{morfiscan}
There exists a morphism in $\mathrm{Mod}(\mathcal{C})$
$$\varphi_{M}:M\longrightarrow \mathbb{L}(M)$$
defined as $[\varphi_{M}]_{C}:=[\psi_{M}]_{C}\circ Y_{C}$ for each $C\in \mathcal{C}$.
\end{proposition}
\begin{proof}
Let $h:C\longrightarrow B$ be a morphism in $\mathcal{C}$. We must check that the following equality holds $$\mathbb{L}(M)(h)\circ [\varphi_{M}]_{C}=[\varphi_{M}]_{B}\circ M(h).$$
Indeed, let $x\in M(C)$ be then we have that $Y_{C}(x):=\eta_{x}:\mathrm{Hom}_{\mathcal{C}}(C,-)\longrightarrow M$. By the commutativity of the diagram $(\ast)$ above, we get that 
$[\varphi_{M}]_{C}(x)=\varphi_{I,M}^{C}(\theta_{I,M}^{C}(\eta_{x}))\in \mathbb{L}(M)(C)$. Since $\theta_{I,M}^{C}(\eta_{x})\in \mathrm{Hom}_{\mathrm{Mod}(\mathcal{C}}(I,M)$, by the diagram (1) before the definition \ref{funtorL}, we get

\begin{align*}
\mathbb{L}(M)(h)\Big([\varphi_{M}]_{C}(x)\Big) & =\varphi_{\alpha^{-1}(I),M}^{B}\Big(\mathrm{Hom}_{\mathrm{Mod}(\mathcal{C})}(\gamma_{I}^{C},M)(\theta_{I,M}^{C}(\eta_{x}))\Big)\\
& =\varphi_{\alpha^{-1}(I),M}^{B}\Big((\theta_{I,M}^{C}(\eta_{x}))\circ \gamma_{I}^{C}\Big)\\
& =\varphi_{\alpha^{-1}(I),M}^{B}\Big((\eta_{x}\circ \delta_{I})\circ \gamma_{I}^{C}\Big),\,\,\,\,\,\,[\text{def. of}\,\,\, \theta_{I,M}^{C}]\\
& =\varphi_{\alpha^{-1}(I),M}^{B}\Big((\eta_{x}\circ \mathrm{Hom}_{\mathcal{C}}(h,-)\circ  \delta_{\alpha^{-1}(I)}\Big),\,\,\,\,[\text{remark}\,\,\,\ref{idealpull}(b) ]
\end{align*}
But, we have that $\eta_{x}\circ \mathrm{Hom}_{\mathcal{C}}(h,-):\mathrm{Hom}_{\mathcal{C}}(B,-)\longrightarrow M$ is such that
$$[\eta_{x}\circ \mathrm{Hom}_{\mathcal{C}}(h,-)]_{B}(1_{B})=[\eta_{x}]_{B}(h)=M(h)(x)\,\,\,\, [\text{def. of Yoneda iso}].$$
Thus, we have that $\eta_{x}\circ \mathrm{Hom}_{\mathcal{C}}(h,-)=\eta_{M(h)(x)}.$
Therefore, we conclude that
$$\mathbb{L}(M)(h)\Big([\varphi_{M}]_{C}(x)\Big) =\varphi_{\alpha^{-1}(I),M}^{B}\Big(\eta_{M(h)(x)}\circ \delta_{\varphi^{-1}(I)}\Big).$$
On the other hand, by the diagram $(\ast)$ above ($B$ instead of $C$ and $\alpha^{-1}(I)$ instead of $I$) and since $Y_{B}\Big(M(h)(x)\Big)=\eta_{M(h)(x)}:\mathrm{Hom}_{\mathcal{C}}(B,-)\longrightarrow M$ we get the equalities
\begin{align*}
[\varphi_{M}]_{B}\Big(M(h)(x)\Big) & =[\psi_{M}]_{B}\Big(Y_{B}\Big(M(h)(x)\Big)\Big)\,\,\,\,\, [\text{def. of}\,\,\,[\varphi_{M}]_{B}]\\
& =\varphi_{\alpha^{-1}(I),M}^{B}\Big(\theta_{\varphi^{-1}(I),M}^{B}\Big(Y_{B}\Big(M(h)(x)\Big)\Big)\Big)\\
& =\varphi_{\alpha^{-1}(I),M}^{B}\Big(\theta_{\varphi^{-1}(I),M}^{B}\Big(\eta_{M(h)(x)}\Big)\Big)\\
& =\varphi_{\alpha^{-1}(I),M}^{B}\Big(\eta_{M(h)(x)}\circ \delta_{\varphi^{-1}(I)}\Big)\,\,\,\,\, [\text{def. of}\,\,\,\theta_{\varphi^{-1}(I),M}^{B}]
\end {align*}
Proving the equality required. So $\varphi_{M}:M\longrightarrow \mathbb{L}(M)$ is a morphism in $\mathrm{Mod}(\mathcal{C})$.
\end{proof}

\begin{proposition}
Consider the functors $1_{\mathrm{Mod}(\mathcal{C})},\mathbb{L} :\mathrm{Mod}(\mathcal{C})\longrightarrow \mathrm{Mod}(\mathcal{C}).$ There is natural transformation
$$\varphi:1_{\mathrm{Mod}(\mathcal{C})}\longrightarrow \mathbb{L}$$
\end{proposition}
\begin{proof}

For $C\in \mathcal{C}$ we have to check that the following diagram commutes
$$\xymatrix{M(C)\ar[r]^{[\varphi_{M}]_{C}}\ar[d]^{\eta_{C}} & \mathbb{L}(M)(C)\ar[d]^{[\overline{\eta}]_{C}}\\
N(C)\ar[r]_{[\varphi_{N}]_{C}} & \mathbb{L}(N)(C).}$$
On one hand we have that
\begin{align*}
\overline{\eta}_{C}\circ [\varphi_{M}]_{C} & = \overline{\eta}_{C}\circ [\psi_{M}]_{C}\circ Y_{C}\\
&  =\overline{\eta}_{C}\circ  \varphi_{I,M}^{C}\circ  \theta_{I,M}^{C}\circ Y_{C}\quad\quad\quad\quad\quad \quad\,\,\,\, [\text{diagram}\,\, 3]\\
& =\varphi_{I,N}^{C}\circ \mathrm{Hom}_{\mathrm{Mod}(\mathcal{C})}(I,\eta)\circ \theta_{I,M}^{C}\circ Y_{C}\quad [\text{diagram}\,\, 2]\
\end{align*}
For $x\in M(C)$ we get $Y_{C}(x):=\alpha_{x}:\mathrm{Hom}_{\mathcal{C}}(C,-)\longrightarrow M$. Then $\theta_{I,M}^{C}(\alpha_{x})=\alpha_{x}\delta_{I}$ with $\delta_{I}:I\longrightarrow \mathrm{Hom}_{\mathcal{C}}(C,-)$ the inclusion. Then $(\overline{\eta}_{C}\circ [\varphi_{M}]_{C})(x)=\varphi_{I,N}^{C}\Big(\eta\circ \alpha_{x}\circ \delta_{I} \Big)$.\\
On the other hand, we have that 
\begin{align*}
[\varphi_{N}]_{C}\circ \eta_{C}= \varphi_{I,N}^{C}\circ \theta_{I,N}^{C}\circ Y_{C}\circ \eta_{C} \quad\quad [\text{Diagram}\,\,3\,\,\text{with N instead of M}].
\end{align*}
For $x\in M(C)$ we get $y:=\eta_{C}(x)\in N(C)$, then $Y_{C}(y)=\gamma_{y}:\mathrm{Hom}_{\mathcal{C}}(C,-)\longrightarrow N$. Then we get  $([\varphi_{N}]_{C}\circ \eta_{C})(x)=\varphi_{I,N}^{C}\Big(\gamma_{y}\circ \delta_{I}\Big)$.
We assert that $\gamma_{y}=\eta\circ \alpha_{x}$.\\ Indeed, we have that for $C'\in \mathcal{C}$
$[\eta\circ \alpha_{x}]_{C'}:\mathrm{Hom}_{\mathcal{C}}(C,C')\longrightarrow N(C')$ is defined as $[\eta\circ \alpha_{x}]_{C'}(f)=[\eta]_{C'}([\alpha_{x}]_{C'}(f))=\eta_{C'}\Big(M(f)(x)\Big)$.\\
On the other side, $[\gamma_{y}]_{C'}(f):=N(f)(y)=N(f)(\eta_{C}(x)).$
But, since $\eta:M\longrightarrow N$ is a natural transformation we have that $N(f)\circ \eta_{C}=\eta_{C'}\circ M(f)$. Then 
$[\gamma_{y}]_{C'}(f):=N(f)(y)=N(f)(\eta_{C}(x))=\eta_{C'}\Big(M(f)(x)\Big)=[\eta\circ \alpha_{x}]_{C'}(f)$. We conclude that $\gamma_{y}=\eta\circ \alpha_{x}$. Then 
$$(\overline{\eta}_{C}\circ [\varphi_{M}]_{C})(x)=\varphi_{I,N}^{C}\Big(\eta\circ \alpha_{x}\circ \delta_{I} \Big)=\varphi_{I,N}^{C}\Big(\gamma_{y}\circ \delta_{I}\Big)=([\varphi_{N}]_{C}\circ \eta_{C})(x)$$
Proving that the required diagram commutes. 
\end{proof}

\begin{proposition}\label{kervarphi}
Let $\mathcal{F}:=\{\mathcal{F}_{C}\}_{C\in \mathcal{C}}$ be a left Gabriel filter in $\mathrm{Mod}(\mathcal{C})$ and consider the morphism $\varphi_{M}:M\longrightarrow \mathbb{L}(M)$ in \ref{morfiscan}. The
$\mathrm{Ker}(\varphi_{M})=t(M)$ where $t$ is de radical associated to the filter $\mathcal{F}$.
\end{proposition}
\begin{proof}
Let $K=\mathrm{Ker}(\varphi_{M})$ be and $\psi:K\longrightarrow M$ the canonical inclusion. Given the filter $\mathcal{F}
:=\{\mathcal{F}_{C}\}_{C\in \mathcal{C}}$ the corresponding 
torsion class is
$$\mathcal{T}_{\mathcal{F}}:=\left\{M\in \mathrm{Mod}(\mathcal{C})\mid  \text{for each}\,\, C\in \mathcal{C},\,\,\, \mathrm{Ann}(x,-)\in \mathcal{F}_{C}\,\,\forall x\in M(C) \right\},$$
and the radical $t$ is defined as:
$$t(M)=\sum_{N\in \mathcal{T}_{\mathcal{F}},N\subseteq M}N.$$ Therefore, $t(M)(C)=\sum_{N\in \mathcal{T}_{\mathcal{F}},N\subseteq M}N(C).$
Let $x\in K(C)$ be, then we get that $0=[\varphi_{M}]_{C}(x)$. Then by definition of $[\varphi_{M}]_{C}$ and by the diagram $3$, we have that  $\varphi_{J,M}^{C}\Big(\theta_{J,M}^{C}(Y_{C}(x))\Big)=0$ for some $J\in \mathcal{F}_{C}$ where  $Y_{C}(x):=\eta_{x}:\mathrm{Hom}_{\mathcal{C}}(C,-)\longrightarrow M$ is such that $[\eta_{x}]_{C}(1_{C})=x$ (Yoneda isomorphism).  Then
$\theta_{J,M}^{C}(Y_{C}(x))=\mathrm{Hom}_{\mathrm{Mod}(\mathcal{C})}(\delta_{J},M)(\eta_{x})=\eta_{x}\circ \delta_{J}:J\longrightarrow M$.\\
Since $\mathcal{F}_{C}$ is a directed set we have that $\varphi_{J,M}^{C}\Big(\theta_{J,M}^{C}Y_{C}(x)\Big)=\varphi_{J,M}^{C}\Big(\eta_{x}\circ \delta_{J}\Big)=0$ implies that there exists  $I\geq J$ ($I\subseteq J$) in $\mathcal{F}_{C}$ such that  $\lambda_{J,I}\Big( \eta_{x}\circ \delta_{J}\Big)=0$ in $\mathrm{Hom}_{\mathrm{Mod}(\mathcal{C})}(I,M)$ (see, lemma 5.30 (ii) in \cite{Rotman}). We recall that 
$\lambda_{J,I}:=\mathrm{Hom}_{\mathrm{Mod}(\mathcal{C})}(\mu_{I,J},M)$, then we get that 
$$0=\lambda_{J,I}\Big( \eta_{x}\circ \delta_{J}\Big)=\mathrm{Hom}_{\mathrm{Mod}(\mathcal{C})}(\mu_{I,J},M)\Big( \eta_{x}\circ \delta_{J}\Big)=\eta_{x}\circ \delta_{J}\circ \mu_{I,J}=\eta_{x}\circ \delta_{I},$$
where $\delta_{I}:I\longrightarrow \mathrm{Hom}_{\mathcal{C}}(C,-)$ is the inclusion. By \ref{biyefilterpretor}, we have that
$$\mathcal{F}_{C}:=\left\{I(C,-)\subseteq \mathrm{Hom}_{\mathcal{C}}(C,-)\mid \frac{\mathrm{Hom}_{\mathcal{C}}(C,-)}{I(C,-)}\in \mathcal{T}_{\mathcal{F}}\right\}.$$
Since $I\in \mathcal{F}_{C}$, we have that $\frac{\mathrm{Hom}(C,-)}{I(C,-)}\in \mathcal{T}_{\mathcal{F}}$ (by the above equality). Since $\eta_{x}\circ \delta_{I}=0$, there exists $\overline{\eta_{x}}:\frac{\mathrm{Hom}_{\mathcal{C}}(C,-)}{I(C,-)}\longrightarrow M$ such that the following diagram commutes
$$\xymatrix{\mathrm{Hom}_{\mathcal{C}}(C,-)\ar[r]^{\eta_{x}}\ar[d]^{\pi_{C}} & M\\
\mathrm{Hom}_{\mathcal{C}}(C,-)/I(C,-)\ar[ur]_{\overline{\eta_{x}}}.}$$ Therefore $[\overline{\eta_{x}}]_{C}\Big(1_{C}+I(C,C)\Big)=x$. Now, let us consider the factorization  of $\overline{\eta_{x}}$
$$\xymatrix{\mathrm{Hom}_{\mathcal{C}}(C,-)/I(C,-)\ar@{->>}[r]  & N\ar@{^{(}->}[r] & M}$$ through its image. Since $\mathcal{T}_{\mathcal{F}}$ is closed under quotients we have that $N\in\mathcal{T}_{\mathcal{F}}$ and $N\subseteq M$. Since $[\overline{\eta_{x}}]_{C}\Big(1_{C}+I(C,C)\Big)=x$ we have that $x\in N(C)$ and therefore $x\in t(M)(C)=\sum_{N\in \mathcal{T}_{\mathcal{F}},N\subseteq M}N(C)$. Proving that $K(C)\subseteq t(M)(C)$.\\
On the other hand, since $t$ is the radical associated to $\mathcal{T}_{\mathcal{F}}$, by \ref{biyepretorprerad} we get that $t(M)\in \mathcal{T}_{\mathcal{F}}$.\\
Now, if $x\in t(M)(C)$ (recall that $t(M)(C)\subseteq M(C)$), then we get that  $I(C,-):=\mathrm{Ann}(x,-)\in\mathcal{F}_{C}$ (since $t(M)\in \mathcal{T}_{\mathcal{F}}$ and the description of $\mathcal{T}_{\mathcal{F}}$ in \ref{filterpretorclas}).  We know that $\mathrm{Ann}(x,C'):=\{f\in \mathrm{Hom}(C,C')\mid M(f)(x)=0\}$. Let $\delta_{I}:I(C,-)\longrightarrow \mathrm{Hom}_{\mathcal{C}}(C,-)$ the inclusion.\\
Consider  $Y_{C}(x):=\eta_{x}:\mathrm{Hom}_{\mathcal{C}}(C,-)\longrightarrow M$ such that $[\eta_{x}]_{C}(1_{C})=x$ (Yoneda isomorphism). We assert that $\eta_{x}\circ \delta_{I}=0$. Indeed, for $C'\in \mathcal{C}$ we have to see that $[\eta_{x}]_{C'}\circ [\delta_{I}]_{C'}=0$. By construction of the Yoneda isomorphism, we have that $[\eta_{x}]_{C'}:\mathrm{Hom}_{\mathcal{C}}(C,C')\longrightarrow M(C')$ satisfies that $[\eta_{x}]_{C'}(f):=M(f)(x)$ $\forall f\in \mathrm{Hom}_{\mathcal{C}}(C,C')$. Then for $f\in I(C,C')$ we have that $[\eta_{x}]_{C'}([\delta_{I}]_{C'}(f))=[\eta_{x}]_{C'}(f)=M(f)(x)=0$ since $[\delta_{I}]_{C'}$ is the inclusion and $f\in I(C,C')=\mathrm{Ann}(x,C')$.  We conclude that $\eta_{x}\circ \delta_{I}=0$.\\
Therefore, by the diagram $3$, we get that 
\begin{align*}
[\varphi_{M}]_{C}(x)= \varphi_{I,M}^{C}\Big(\theta_{I,M}^{C}\Big(Y_{C}(x)\Big)\Big)= \varphi_{I,M}^{C}\Big(\theta_{I,M}^{C}\Big(\eta_{x}\Big)\Big)= \varphi_{I,M}^{C}\Big(\eta_{x}\circ \delta_{I}\Big)=0.
\end{align*}
This, implies by definition of $K$ that $x\in K(C)$. Therefore $t(M)(C)\subseteq K(C)$ and then we conclude that $K=t(M)$.
\end{proof}

\begin{proposition}\label{L=0siMtor}
Let $M\in \mathrm{Mod}(\mathcal{C})$ be. Then $M$ is an $\mathcal{F}$-torsion module if and only if $\mathbb{L}(M)=0$.
\end{proposition}
\begin{proof}
Let us suppose that $M$ is an $\mathcal{F}$-torsion module. We known that an element  $w\in \mathbb{L}(M)(C)$ is of the form $w=\varphi_{I,M}^{C}(\beta)$ for some ideal $I(C,-)\in \mathcal{F}_{C}$ and $\beta:I(C,-)\longrightarrow M$ (see, lemma 5.30 (i) in \cite{Rotman}). Let $v:K(C,-)\longrightarrow I(C,-)$ the kernel of $\beta$. We are going to prove that $K(C,-)\in \mathcal{F}_{C}$.\\
Let $C'\in \mathcal{C}$ and $f\in I(C,C')$, then $\beta_{C'}(f)\in M(C')$. Since $M$ is an $\mathcal{F}$-torsion module by hipothesis, we have that $\mathrm{Ann}\Big(\beta_{C'}(f),-\Big)\in \mathcal{F}_{C'}$ (see \ref{Ftorsion}). We recall that, $\mathrm{Ann}\Big(\beta_{C'}(f),-\Big)(X)=\{h:C'\longrightarrow X\mid M(h)(\beta_{C'}(f))=0 \}$ for all $X\in \mathcal{C}$.\\
Now, we have the Yoneda isomorphism $\mathrm{Hom}_{\mathrm{Mod}(\mathcal{C})}(\mathrm{Hom}_{\mathcal{C}}(C',-),I(C,-))\simeq I(C,C')$. So, $f$ induces a morphism $\eta_{f}:\mathrm{Hom}_{\mathcal{C}}(C',-)\longrightarrow I(C,-)$ in $\mathrm{Mod}(\mathcal{C})$ such that
$[\eta_{f}]_{X}:\mathrm{Hom}_{\mathcal{C}}(C',X)\longrightarrow I(C,X)$ is defined as $[\eta_{f}]_{X}(h):=I(C,h)(f)=hf$ $\forall h\in \mathrm{Hom}_{\mathcal{C}}(C',X)$ and for all $X\in \mathcal{C}$.\\
Then we have that
$\beta\circ \eta_{f}:\mathrm{Hom}_{\mathcal{C}}(C',-)\longrightarrow  M$
which by Yoneda isomorphism it corresponds to $[\beta\circ \eta_{f}]_{C'}(1_{C'})\in M(C')$,
but $[\beta\circ \eta_{f}]_{C'}(1_{C'})=\beta_{C'}(f)$. Now, considering $\beta_{C'}(f)\in M(C')$ by Yoneda isomorphism it corresponds to
$\eta:\mathrm{Hom}_{\mathcal{C}}(C',-)\longrightarrow  M$ such that for every $X$ we have that
$\eta_{X}(h)=M(h)(\beta_{C'}(f))$ for all $h\in\mathrm{Hom}_{\mathcal{C}}(C',X)$.  Then we conclude that
$$\eta=\beta\circ \eta_{f}.$$
We assert that
$$(\beta\circ \eta_{f})|_{\mathrm{Ann}(\beta_{C'}(f),-)}=0$$
Indeed, we have to see that for all $X$
$$\Big([\beta\circ \eta_{f}]_{X}\Big)|_{\mathrm{Ann}(\beta_{C'}(f),X)}=0$$
Indeed, let $h:C'\longrightarrow X$ in $\mathrm{Ann}(\beta_{C'}(f),X)$. Then we have that $([\beta\circ \eta_{f}]_{X}(h)=\eta_{X}(h)=M(h)(\beta_{C'}(f))=0$. This proves that $(\beta\eta_{f})|_{\mathrm{Ann}(\beta_{C'}(f),-)}=0.$\\
Let $u_{C'}:\mathrm{Ann}(\beta_{C'}(f),-)\longrightarrow \mathrm{Hom}_{\mathcal{C}}(C',-)$ be the inclusion. Since $v=\mathrm{Ker}(\beta)$, there exists a unique morphism  $\psi_{C,f}^{C'}$ such that the diagram commutes
$$(\ast):\quad \xymatrix{\mathrm{Ann}(\beta_{C'}(f),-)\ar[dr]^{\eta_{f}\circ u_{C'}}\ar[d]_{\psi_{C,f}^{C'}} & \\
K(C,-)\ar[r]^{v} &  I(C,-)\ar[r]^{\beta} & M.}$$
The family of morphisms $\{\psi_{C,f}^{C'}\}_{f\in I(C,C')}$ induces a a morphism
$$\Psi:\bigoplus_{f\in I(C,C')}\mathrm{Ann}(\beta_{C'}(f),-)\longrightarrow K(C,-)$$
We define $J(C,-):=\mathrm{Im}(\Psi)$.\\
We assert that
$$\mathrm{Ann}(\beta_{C'}(f),-)\subseteq  \Big(J(C,-): f\Big)$$
Indeed, since $f\in I(C,-)(C')=I(C,C')$ and $J(C,-)\subseteq I(C,-)$ we have by definition (see \ref{anulador}(a))  that
$$\Big(J(C,-): f\Big)(X)=\Big(J(C,X):f\Big)=\{h\in \mathrm{Hom}_{\mathcal{C}}(C',X)\mid I(C,h)(f)=hf\in J(C,X)\}$$ 
Now, by definition of $J(C,-)$, we have that 
$$J(C,X):=\sum_{f\in I(C,C')}\mathrm{Im}([\psi_{C,f}^{C'}]_{X})$$\\
Now, since $v$  and $u_{C'}$ are the inclusions, for $h\in \mathrm{Ann}(\beta_{C'}(f),X)$  we get that $[\psi_{C,f}^{C'}]_{X}(h)=v_{X}\Big([\psi_{C,f}^{C'}]_{X}(h)\Big)=[\eta_{f}]_{X}\Big([u_{C'}]_{X}(h)\Big)=[\eta_{f}]_{X}(h)=hf$ (because of the diagram $(\ast)$ and the definition of $\eta_{f}$).\\
We conclude that $hf\in \mathrm{Im}([\psi_{C,f}^{C'}]_{X})\subseteq J(C,X)$. This tell us that $h\in \Big(J(C,-):f\Big)(X)$.\\ Therefore
$\mathrm{Ann}(\beta_{C'}(f),-)\subseteq  \Big(J(C,-): f\Big)$  $\forall f\in I(C,C')$. Since $\mathrm{Ann}(\beta_{C'}(f),-)\in \mathcal{F}_{C'}$ ($M$ is an $\mathcal{F}$-torsion module), by $T_{1}$ in definition \ref{filterdef}, we conclude that
$$\Big(J(C,-):f\Big)\in \mathcal{F}_{C'}\,\,\, \forall f\in I(C,C').$$
Since $I(C,-)\in \mathcal{F}_{C}$, by $T_{4}$ we conclude that $J(C,-)\in \mathcal{F}_{C}$. Since $J(C,-)\subseteq K(C,-)$ by $T_{1}$ we conclude that $K(C,-)\in \mathcal{F}_{C}$.\\
Let us define $\mu_{K,I}=v:K(C,-)\longrightarrow I(C,-)$ as the canonical inclusion. Thus, have that $\lambda_{I,K}(\beta)=\mathrm{Hom}_{\mathrm{Mod}(\mathcal{C})}\Big(\mu_{K,I},M\Big)(\beta)=\beta \circ \mu_{I,K}=\beta \circ v=0$, since $v:K(C,-)\longrightarrow I(C,-)$ is the kernel of $\beta$. By lemma 5.30 (ii) in \cite{Rotman}, we conclude that $w=\varphi_{I,M}^{C}(\beta)=0$  in
$\mathbb{L}(M)(C)$. Therefore $\mathbb{L}(M)(C)=0$ and thus $\mathbb{L}(M)=0$.\\
Conversely, suppose that $\mathbb{L}(M)=0$. Then by \ref{kervarphi}, we have that $M=\mathrm{Ker}(\varphi_{M})=t(M)$. Proving that $M$ is an $\mathcal{F}$-torsion module.
\end{proof}

\begin{proposition}\label{cuadradochido}
Let $w=\varphi_{I,M}^{C}(\beta)\in \mathbb{L}(M)(C)$  for some ideal $I(C,-)\in \mathcal{F}_{C}$ and $\beta:I(C,-)\longrightarrow M$. Then the following diagrama
commutes
$$\xymatrix{I(C,-)\ar[r]^{\delta_{I}}\ar[d]_{\beta} & \mathrm{Hom}_{\mathcal{C}}(C,-)\ar[d]^{\psi}\\
M\ar[r]_{\varphi_{M}} & \mathbb{L}(M)}$$
where $\delta_{I}$ is the canonical inclusion and  $\psi$ is the natural transformation corresponding to $w=\varphi_{I,M}^{C}(\beta)\in \mathbb{L}(M)(C)$ via the Yoneda isomorphism $Y:\mathbb{L}(M)(C)\longrightarrow \mathrm{Hom}_{\mathrm{Mod}(\mathcal{C})}\Big(\mathrm{Hom}_{\mathcal{C}}(C,-),\mathbb{L}(M)\Big).$
\end{proposition}
\begin{proof}
First, we recall that $\psi:\mathrm{Hom}_{\mathcal{C}}(C,-)\longrightarrow \mathbb{L}(M)$ is such that for $C'\in \mathcal{C}$
$\psi_{C'}:\mathrm{Hom}_{\mathcal{C}}(C,C')\longrightarrow \mathbb{L}(M)(C')$ is defined as:  $\psi_{C'}(f):=\mathbb{L}(M)(f)(w)$ for all $f\in \mathrm{Hom}_{\mathcal{C}}(C,C')$.\\
Now, let $C'\in \mathcal{C}$ be, we have to show that the following equality holds
$$\psi_{C'}\circ [\delta_{I}]_{C'}=[\varphi_{M}]_{C'}\circ \beta_{C'}.$$
Indeed, let $f\in I(C,C')$ be, since $[\delta_{I}]_{C'}$ is the inclusion and by the diagram (1) before definition \ref{funtorL} we have that
\begin{align*}
\psi_{C'}([\delta_{I}]_{C'}(f))=
\psi_{C'}(f)=\mathbb{L}(M)(f)\big(\varphi_{I,M}^{C}(\beta)\big) & =\varphi_{\alpha^{-1}(I),M}^{C'}(\mathrm{Hom}(\gamma^{C}_{I},M)(\beta))\\
& =\varphi_{\alpha^{-1}(I),M}^{C'}(\beta\circ \gamma^{C}_{I})
\end{align*}
where $\gamma^{C}_{I}:\alpha^{-1}(I(C,-))\longrightarrow I(C,-)$ is such that the following diagram commutes
 $$\xymatrix{\alpha^{-1}(I(C,-))\ar[rr]^{\gamma^{C}_{I}}\ar[d]_{\delta'_{\alpha^{-1}(I)}}  & & I(C,-)\ar[d]^{\delta_{I}}\\
\mathrm{Hom}_{\mathcal{C}}(C',-)\ar[rr]_{\alpha=\mathrm{Hom}_{\mathcal{C}}(f,-)} && \mathrm{Hom}_{\mathcal{C}}(C,-)}$$
with $\gamma^{C}_{I}:=\mathrm{Hom}_{\mathcal{C}}(f,-)|_{(I(C,-):f)}$ and $\delta_{I}:I(C,-)\longrightarrow \mathrm{Hom}_{\mathcal{C}}(C,-)$ is the inclusion (see (see \ref{idealpull}(b)).\\
On the other hand, consider $y:=\beta_{C'}(f)\in M(C')$, via the Yoneda isomorphism $Y_{C'}:M(C')\longrightarrow \mathrm{Hom}_{\mathrm{Mod}(\mathcal{C})}(\mathrm{Hom}_{\mathcal{C}}(C',-),M)$ it corresponds to the natural transformation $\eta_{y}:\mathrm{Hom}_{\mathcal{C}}(C',-)\longrightarrow M$ such that 
$$[\eta_{y}]_{X}(h)=M(h)(\beta_{C'}(f))\,\,\,\forall X\in \mathcal{C}\,\,\, \forall h\in \mathrm{Hom}_{\mathcal{C}}(C',X).$$
Since $\alpha^{-1}(I(C,-))\in \mathcal{F}_{C'}$ (by $T_{3}$) and $\eta_{y}:\mathrm{Hom}_{\mathcal{C}}(C',-)\longrightarrow M$,  by definition of $\varphi_{M}$ we have that
$$[\varphi_{M}]_{C'}(\beta_{C'}(f))=\varphi_{\alpha^{-1}(I),M}^{C'}\Big(\theta_{\alpha^{-1}(I),M}^{C'}(Y_{C'}(\beta_{C'}(f))\Big)=\varphi_{\alpha^{-1}(I),M}^{C'}\Big(\theta_{\alpha^{-1}(I),M}^{C'}(\eta_{y})\Big).$$ 
In order to show that 
$$\psi_{C'}([\delta_{I}]_{C'}(f))=[\varphi_{M}]_{C'}(\beta_{C'}(f)),$$ is enough to show that $\beta\circ \gamma^{C}_{I}=\theta_{\alpha^{-1}(I),M}^{C'}(\eta_{y})$.\\
But, if we denote by $\delta'_{\alpha^{-1}(I)}:\alpha^{-1}(I)\longrightarrow \mathrm{Hom}_{\mathcal{C}}(C',-)$ we have that $\theta_{\alpha^{-1}(I),M}^{C'}(\eta_{y})=\mathrm{Hom}_{\mathrm{Mod}(\mathcal{C})}\Big(\delta'_{\alpha^{-1}(I)},M\Big)(\eta_{y})=\eta_{y}\circ \delta'_{\alpha^{-1}(I)}$.
So, for $X\in \mathcal{C}$ we have to show the following diagram commutes
$$\xymatrix{\alpha^{-1}(I(C,-))(X)\ar[rr]^{[\gamma^{C}_{I}]_{X}}\ar[d]_{[\delta'_{\alpha^{-1}(I)}]_{X} } & & I(C,X)\ar[d]^{[\beta]_{X}}\\
\mathrm{Hom}_{\mathcal{C}}(C',X)\ar[rr]_{[\eta_{y}]_{X}} && M(X).}$$ 
Indeed, let $h:C'\longrightarrow X$ in $\alpha^{-1}(I(C,-))(X)$, since $[\delta'_{\alpha^{-1}(I)}]_{X}$ is the inclusion,  we have that $[\eta_{y}]_{X}\Big([\delta'_{\alpha^{-1}(I)}]_{X}(h)\Big)=[\eta_{y}]_{X}(h):=M(h)(\beta_{C'}(f))$ (see def. of $\eta_{y}$ above).\\
On the other hand, $\beta_{X}([\gamma^{C}_{I}]_{X}(h)):=\beta_{X}(hf)$ (since $\gamma^{C}_{I}:=\mathrm{Hom}_{\mathcal{C}}(f,-)|_{(I(C,-):f)}$).\\
Now, since $\beta:I(C,-)\longrightarrow M$ is a natural transformation, by considering $h:C'\longrightarrow X$  we have that
$M(h)\circ \beta_{C'}=\beta_{X}\circ I(C,h).$\\
Then $M(h)\Big(\beta_{C'}(f)\Big)=\beta_{X}(I(C,h)(f))=\beta_{X}(hf)$.
This shows that $\beta\circ \gamma_{I}^{C}=\eta_{y} \circ \delta'_{\alpha^{-1}(I)}.$
Proving that $\psi_{C'}([\delta_{I}]_{C'}(f))=[\varphi_{M}]_{C'}(\beta_{C'}(f))$. Therefore the required diagram commutes.
\end{proof}

\begin{proposition}\label{Cokervarphi}
For each $M\in \mathrm{Mod}(\mathcal{C})$ we have that 
$\mathrm{Coker}(\varphi_{M})$ is an $\mathcal{F}$-torsion module.
\end{proposition}
\begin{proof}
Let us consider $N:=\mathrm{Im}(\varphi_{M})$ (that is, $N(C):=\mathrm{Im}([\varphi_{M}]_{C})$ for $C\in \mathcal{C}$) and $u:N\longrightarrow \mathbb{L}(M)$ the inclusion. Then we have that $\pi:\mathbb{L}(M)\longrightarrow \mathbb{L}(M)/N$ is the cokernel of $\varphi_{M}$, where $(\mathbb{L}(M)/N)(C):=\mathbb{L}(M)(C)/N(C)$ for $C\in \mathcal{C}$. Let $\overline{w}:=w+N(C)\in \mathbb{L}(M)(C)/N(C)$ with $w\in \mathbb{L}(M)(C)$. We known that $w=\varphi_{I,M}^{C}(\beta)$ for some ideal $I(C,-)\in \mathcal{F}_{C}$ and $\beta:I(C,-)\longrightarrow M$. 
By \ref{cuadradochido}, for each $C'\in \mathcal{C}$ we have the following commutative diagram 
$$\xymatrix{I(C,C')\ar[r]^{[\delta_{I}]_{C'}}\ar[d]_{\beta_{C'}} & \mathrm{Hom}_{\mathcal{C}}(C,C')\ar[d]^{\psi_{C'}}\\
M(C')\ar[r]_{[\varphi_{M}]_{C'}} & \mathbb{L}(M)(C')\ar[r] & \mathbb{L}(M)(C')/N(C')\ar[r] & 0. }$$
For $f\in I(C,C')$ we get that $\psi_{C'}(f)=\psi_{C'}([\delta_{I}]_{C'}(f))=[\varphi_{M}]_{C'}(\beta_{C'}(f))\in \mathrm{Im}([\varphi_{M}]_{C'})$. But by definition of $\psi_{C'}$ we have that
$\psi_{C'}(f)=\mathbb{L}(M)(f)\big(\varphi_{I,M}^{C}(\beta)\big)=\mathbb{L}(M)(f)\big( w\big)$. Therefore we conclude that
$$(\ast):\quad \Big(\mathbb{L}(M)(f)\Big)\big( w\big)=[\varphi_{M}]_{C'}(\beta_{C'}(f))\in  N(C')=\mathrm{Im}([\varphi_{M}]_{C'})\quad \forall f\in I(C,C').$$
Since $\overline{w}\in (\mathbb{L}(M)/N)(C)$ we have that
$\mathrm{Ann}(\overline{w},-)$ is a left ideal of $\mathrm{Hom}_{\mathcal{C}}(C,-)$ defined as
\begin{align*}
\mathrm{Ann}(\overline{w},-)(C'): & =\{f\in \mathrm{Hom}_{\mathcal{C}}(C,C')\mid \Big((\mathbb{L}(M)/N)(f)\Big)(\overline{w})=0\}\\
 & = \{f\in \mathrm{Hom}_{\mathcal{C}}(C,C')\mid \Big(\mathbb{L}(M)(f)\Big)(w)\in N(C')\}
\end{align*}
By the assertion given in $(\ast)$ above, we have that $I(C,-)\subseteq \mathrm{Ann}\big(\overline{w},-\big)$. Since $I\in \mathcal{F}_{C}$ and $\mathcal{F}$ is a Gabriel filter, we have that $\mathrm{Ann}\big(\overline{w},-\big)\in \mathcal{F}_{C}$. This proves that $\mathbb{L}(M)/N$ is an $\mathcal{F}$-torsion module (see def. \ref{Ftorsion}).
\end{proof}

\section{Gabriel localization}\label{sec4}
Now, we recall the following construction: given a radical $t:\mathrm{Mod}(\mathcal{C})\longrightarrow  \mathrm{Mod}(\mathcal{C})$ we can construct a fucntor
$q_{t}:\mathrm{Mod}(\mathcal{C})\longrightarrow  \mathrm{Mod}(\mathcal{C})$ defined as $q_{t}(M):=\frac{M}{t(M)}$ for all $M\in \mathrm{Mod}(\mathcal{C})$.

\begin{definition}
Let $\mathcal{C}$ be a preadditive category and $\mathcal{F}:=\{\mathcal{F}_{C}\}_{C\in \mathcal{C}}$ be a left Gabriel filter in the category $\mathcal{C}$. The $\textbf{Gabriel localisation functor}$  with respect to $\mathcal{F}$ is the functor
$$\mathbb{G}:=\mathbb{L}\circ q_{t}:\mathrm{Mod}(\mathcal{C})\longrightarrow  \mathrm{Mod}(\mathcal{C}),$$ where $t$ is the radical associated to the filter $\mathcal{F}$. The $\textbf{Gabriel localisation}$ of $M$ with respect to $\mathcal{F}$ is the $\mathcal{C}$-module
$$\mathbb{G}(M):=\mathbb{L}\Big (\frac{M}{t(M)}\Big).$$
\end{definition}

\begin{definition}
We define $\Delta_{M}:M\longrightarrow \mathbb{G}(M)$ as the composition 
$$\xymatrix{M\ar[r]^{\pi_{M}} & \frac{M}{t(M)}\ar[r]^{\varphi_{\frac{M}{t(M)}}} & \mathbb{L}(\frac{M}{t(M)}).}$$
\end{definition}

\begin{proposition}\label{isoLL}
There exists an isomorphism
$$\mathbb{G}(M)\simeq \mathbb{L}\mathbb{L}(M).$$
\end{proposition}
\begin{proof}
Consider the exact sequence
$$\xymatrix{0\ar[r] & t(M)\ar[r] & M\ar[r]^{\varphi_{M}} & \mathbb{L}(M)\ar[r] & \mathrm{Coker}(\varphi_{M})\ar[r] & 0}$$
Consider $\xymatrix{M\ar@{->>}[r]^{\pi_{M}} &\frac{M}{t(M)}\ar@{^{(}->}[r]^{\gamma_{M}} & \mathbb{L}(M)}$ the factorization through its image of $\varphi_{M}$.
Then we have the exact sequence
$$\xymatrix{0\ar[r] & M/t(M)\ar[r]^{\gamma_{M}} & \mathbb{L}(M)\ar[r] & \mathrm{Coker}(\varphi_{M})\ar[r] & 0.}$$
Applying $\mathbb{L}$ we get the exact sequence
$$\xymatrix{0\ar[r] & M/t(M)\ar[r]^{\gamma_{M}}\ar[d]^{\varphi_{\frac{M}{t(M)}}} & \mathbb{L}(M)\ar[r]\ar[d]^{\varphi_{\mathbb{L}(M)}} & \mathrm{Coker}(\varphi_{M})\ar[r]\ar[d]^{\varphi_{\mathrm{Coker}(\varphi_{M})}} & 0\\
0\ar[r] & \mathbb{L}(M/t(M))\ar[r]^{\mathbb{L}(\gamma_{M})} & \mathbb{L}(\mathbb{L}(M))\ar[r] & \mathbb{L}(\mathrm{Coker}(\varphi_{M})).}$$
Since $\mathrm{Coker}(\varphi_{M})$ is an $\mathcal{F}$-torsion module, we have that
$ \mathbb{L}(\mathrm{Coker}(\varphi_{M}))=0$. Proving that $\mathbb{L}(\gamma_{M})$ is an isomorphism. That is $\mathbb{G}(M)\simeq \mathbb{L}(\mathbb{L}(M)).$
\end{proof}

\begin{remark}\label{otherdescLL}
\begin{enumerate}
\item [(a)]
By the proof of \ref{isoLL} we get the following commutative diagram
$$\xymatrix{M\ar[d]^{\pi_{M}}\ar[r]^{\varphi_{M}} & \mathbb{L}(M)\ar@{=}[d]\\
M/t(M)\ar[r]^{\gamma_{M}}\ar[d]^{\varphi_{\frac{M}{t(M)}}} & \mathbb{L}(M)\ar[d]^{\varphi_{\mathbb{L}(M)}}\\
\mathbb{L}(M/t(M))\ar[r]^{\mathbb{L}(\gamma_{M})} & \mathbb{L}(\mathbb{L}(M)).}$$
where $\mathbb{L}(\gamma_{M})$ is an isomorphism. Then we we could have defined the $\Delta_{M}$ as the composition
$\varphi_{\mathbb{L}(M)}\circ \varphi_{M}: M \longrightarrow \mathbb{L}\mathbb{L}(M).$

\item [(b)] We have that $\varphi_{\frac{M}{t(M)}}:\frac{M}{t(M)}\longrightarrow \mathbb{L}(\frac{M}{t(M)})$ is a monomorphism. Then the composition  $\xymatrix{M\ar@{->>}[r]^{\pi_{M}} &\frac{M}{t(M)}\ar@{^{(}->}[r]^{\varphi_{\frac{M}{t(M)}}}& \mathbb{L}(\frac{M}{t(M)})}$ is the factorization of $\Delta_{M}$ through its image. Indeed, this follows from \ref{kervarphi} and the fact that $t(\frac{M}{t(M)})=0$, since $t$ is a radical.
\end{enumerate}
\end{remark}

\begin{corollary}\label{kerDelta}
We have exact sequence
$$\xymatrix{0\ar[r] & t(M)\ar[r] & M\ar[r]^{\Delta_{M}} & \mathbb{G}(M)\ar[r] & \mathrm{Coker}(\varphi_{\frac{M}{t(M)}})\ar[r] & 0}$$
In particular, $\mathrm{Ker}(\Delta_{M})$ and $\mathrm{Coker}(\Delta_{M})$ are of $\mathcal{F}$-torsion.
\end{corollary}
\begin{proof}
This follows from \ref{otherdescLL}(b), \ref{Cokervarphi} and \ref{kervarphi}.
\end{proof}

\begin{proposition}
For each module we have that $\mathbb{G}(M)\simeq \mathbb{G}(\frac{M}{t(M)})$.
\end{proposition}
\begin{proof}
We have that $t(M/t(M))=0$, so, we have that $q_{t}(M/t(M))=M/t(M)$.
So, we have that $\mathbb{G}(M/t(M))=\mathbb{L}(q_{t}(M/t(M)))=\mathbb{L}(M/t(M))=\mathbb{L}q_{t}(M)=\mathbb{G}(M).$
\end{proof}

\begin{definition}
A $\mathcal{C}$-module $M$ is $\mathcal{F}$-$\textbf{closed}$ if
$$\theta_{I,M}^{C}:\mathrm{Hom}_{\mathrm{Mod}(\mathcal{C})}(\mathrm{Hom}_{\mathcal{C}}(C,-),M)\longrightarrow \mathrm{Hom}_{\mathrm{Mod}(\mathcal{C})}(I,M)$$ is an isomorphism for all $C\in \mathcal{C}$ and $I\in \mathcal{F}_{C}$ (recall that $\theta_{I,M}^{C}:=\mathrm{Hom}_{\mathrm{Mod}(\mathcal{C})}(\delta_{I},M)$ where $\delta_{I}:I\longrightarrow \mathrm{Hom}_{\mathcal{C}}(C,-)$ is the inclusion).
\end{definition}

\begin{proposition}\label{Flibreclode}
Let $M\in \mathrm{Mod}(\mathcal{C})$ be a $\mathcal{F}$-closed module. Let $C\in \mathcal{C}$ be and $x\in M(C)$ such that $\mathrm{Ann}(x,-)\in \mathcal{F}_{C}$, then $x=0$. In particular $t(M)=0$.
\end{proposition}
\begin{proof}
By Yoneda's Lemma we have a morphism 
$\eta_{x}:\mathrm{Hom}_{\mathcal{C}}(C,-)\longrightarrow M$ such that
$[\eta_{x}]_{B}(f)=M(f)(x)$ for all $f\in \mathrm{Hom}_{\mathcal{C}}(C,B)$. Let $I(C,-):=\mathrm{Ann}(x,-)$ and $\delta_{I}:I\longrightarrow \mathrm{Hom}_{\mathcal{C}}(C,-)$ the canonical inclusion. Since $M$ is $\mathcal{F}$-closed we have an isomorphism
$$\theta_{I,M}^{C}:\mathrm{Hom}_{\mathrm{Mod}(\mathcal{C})}(\mathrm{Hom}_{\mathcal{C}}(C,-),M)\longrightarrow \mathrm{Hom}_{\mathrm{Mod}(\mathcal{C})}(I,M)$$
But $\theta_{I,M}^{C}(\eta_{x})=\eta_{x}\circ \delta_{I}:I\longrightarrow M$ satisfies that $[\eta_{x}\circ \delta_{I}]_{B}(f)=[\eta_{x}]_{B}(f)=M(f)(x)=0$ for all $B\in \mathcal{C}$ and for all $f\in \mathrm{Hom}_{\mathcal{C}}(C,B)$, since $f\in I(C,B)=\mathrm{Ann}(x,-)(B)$. Thus, we have that $\eta_{x}\circ \delta_{I}=0$ and since $\theta_{I,M}^{C}$ is an isomorphism we have that $\eta_{x}=0$ and therefore by Yoneda's Lemma we conclude that $x=0$.\\
Now, since $t(M)=\sum_{N\in \mathcal{T}_{\mathcal{F}},N\subseteq M}N,$ we conclude that $t(M)=0$.
\end{proof}

\begin{corollary}\label{Closeddeltaiso}
If $M$ is an $\mathcal{F}$-closed $\mathcal{C}$-module, then $\Delta_{M}:M\longrightarrow \mathbb{G}(M)$ is an isomorphism.
\end{corollary}
\begin{proof}
By \ref{Flibreclode} we conclude that $\pi_{M}=1$, $M=M/t(M)$ and $\varphi_{M}=\varphi_{\frac{M}{t(M)}}$. Now by the diagram $3$, and the definition of $\mathcal{F}$-closed we have that $\varphi_{M}$ is an isomorphism. Proving that $\Delta_{M}$ is an isomorphism.
\end{proof}

\begin{proposition}\label{igualrest}
Let $\mathcal{F}=\{\mathcal{F}_{C}\}_{C\in \mathcal{C}}$ be a left Gabriel filter.  Let $I,J\in \mathcal{F}_{C}$ with $\mu_{I,J}:I\longrightarrow J$ the inclusion and $M$ a torsion-free module (that is, $t(M)=0$). Let $f,g:J\longrightarrow M$ be a morphism such that $f\circ \mu_{I,J}=g\circ \mu_{I,j}$, then $f=g$.
\end{proposition}
\begin{proof}
Consider the exact sequence
$$\xymatrix{0\ar[r] & I(C,-)\ar[r]^{\mu_{I,J}} & J(C,-)\ar[r]^{\pi} & \frac{J(C,-)}{I(C,-)}\ar[r] & 0}$$
Since $(f-g)\circ \mu_{I,J}=0$, there exists $\theta:\frac{J(C,-)}{I(C,-)}\longrightarrow M$ such that $f-g=\theta\pi$.\\
Let $\overline{w}\in \frac{J(C,C')}{I(C,C')}$ be with $w\in J(C,C')\subseteq 	\mathrm{Hom}_{\mathcal{C}}(C,C')$, then we have that
$\mathrm{Ann}(\overline{w},-)$ is a left ideal of $\mathrm{Hom}_{\mathcal{C}}(C',-)$ and we easy get that
\begin{align*}
\mathrm{Ann}(\overline{w},-)(X)= \{f\in \mathrm{Hom}_{\mathcal{C}}(C',X)\mid f w\in I(C,X)\}
\end{align*}
Then $\mathrm{Ann}(\overline{w},-)=\Big(I(C,-):w\Big)$. By $T_{3}$ we have that $\mathrm{Ann}(\overline{w},-)\in \mathcal{F}_{C'}$. Thus, we get that $\frac{J(C,-)}{I(C,-)}\in \mathcal{T}_{\mathcal{F}}$ (that is $\frac{J(C,-)}{I(C,-)}$ is a $\mathcal{F}$-torsion module).
Since $\frac{J(C,-)}{I(C,-)}$ is a $\mathcal{F}$-torsion module (equivalently $t(\frac{J(C,-)}{I(C,-)})=0$) and $M$ is a torsion free module, we have that $\theta:\frac{J(C,-)}{I(C,-)}\longrightarrow M$ is the morphism zero and then we have that $f-g=\theta\pi=0$. Proving that $f=g$.
\end{proof}

\begin{proposition}\label{Lrestrictorsion}
Let $\mathcal{F}=\{\mathcal{F}_{C}\}_{C\in \mathcal{C}}$ be a left Gabriel filter and $M\in \mathrm{Mod}(\mathcal{C})$ and $t$ the radical associated to $\mathcal{F}$. If $t(M)=0$, then $t(\mathbb{L}(M))=0$.
\end{proposition}
\begin{proof}
Consider the filter $\mathcal{F}:=\{\mathcal{F}_{C}\}_{C\in \mathcal{C}}$ in $\mathcal{C}$. By \ref{biyeGafilterpretor}  we have the torsion class
$$\mathcal{T}_{\mathcal{F}}:=\left\{M\in \mathrm{Mod}(\mathcal{C})\mid  \text{for each}\,\, C\in \mathcal{C},\,\,\, \mathrm{Ann}(x,-)\in \mathcal{F}_{C}\,\,\forall x\in M(C) \right\}.$$
By \ref{biyepretorprerad}, we have a radical $t$ associated to $\mathcal{T}_{\mathcal{F}}$. Then we get that
$$t(\mathbb{L}(M))=\sum_{N\in \mathcal{T}_{\mathcal{F}},N\subseteq \mathbb{L}(M)}N.$$
Let $w\in t(\mathbb{L}(M))(C)\subseteq \mathbb{L}(M)(C)$. Then we get that  $J(C,-):=\mathrm{Ann}(w,-)\in\mathcal{F}_{C}$ (since $t(\mathbb{L}(M))\in \mathcal{T}_{\mathcal{F}}$ and the description of $\mathcal{T}_{\mathcal{F}}$ in \ref{filterpretorclas}).  We recall  that $\mathrm{Ann}(w,C'):=\{f\in \mathrm{Hom}(C,C')\mid \mathbb{L}(M)(f)(w)=0\}$.\\
We know that $w=\varphi_{I,M}^{C}(\beta)\in \mathbb{L}(M)(C)$  for some ideal $I(C,-)\in \mathcal{F}_{C}$ and $\beta:I(C,-)\longrightarrow M$. Then by \ref{cuadradochido}, we have that
$$\psi\circ \delta_{I}=\varphi_{M}\circ \beta,$$
where $\delta_{I}$ is the canonical inclusion and  $\psi$ is the natural transformation corresponding to $w=\varphi_{I,M}^{C}(\beta)\in \mathbb{L}(M)(C)$ via the Yoneda isomorphism $Y:\mathbb{L}(M)(C)\longrightarrow \mathrm{Hom}_{\mathrm{Mod}(\mathcal{C})}\Big(\mathrm{Hom}_{\mathcal{C}}(C,-),\mathbb{L}(M)\Big)$. Then we have
$$\xymatrix{J(C,-)\cap I(C,-)\ar[r]^{\epsilon_{J}}\ar[d]_{\epsilon_{I}} & J(C,-)\ar[d]^{\delta_{J}}\\
I(C,-)\ar[r]^{\delta_{I}}\ar[d]_{\beta} & \mathrm{Hom}_{\mathcal{C}}(C,-)\ar[d]^{\psi}\\
M\ar[r]_{\varphi_{M}} & \mathbb{L}(M).}$$ 
Then for $C'\in \mathcal{C}$ we have the diagram

$$\xymatrix{J(C,C')\cap I(C,C')\ar[r]^(.6){[\epsilon_{J}]_{C'}}\ar[d]_{[\epsilon_{I}]_{C'}} & J(C,C')\ar[d]^{[\delta_{J}]_{C'}}\\
I(C,C')\ar[r]^{[\delta_{I}]_{C'}}\ar[d]_{\beta_{C'}} & \mathrm{Hom}_{\mathcal{C}}(C,C')\ar[d]^{\psi_{C'}}\\
M(C')\ar[r]_{[\varphi_{M}]_{C'}} & \mathbb{L}(M)(C'),}$$
where $[\delta_{I}]_{C'}$, $[\delta_{J}]_{C'}$, $[\epsilon_{I}]_{C'}$ and $[\epsilon_{J}]_{C'}$ are the inclusions as sets. Then for $f\in J(C,C')	\cap I(C,C')$ we get that 
$$\psi_{C'}([\delta_{J}]_{C'}[\epsilon_{J}]_{C'}(f))=\psi_{C'}(f)=\mathbb{L}(M)(f)(w)=0$$
since $f\in \mathrm{Ann}(w,C')$. Therefore we obtain that
$\varphi_{M}\circ \beta\circ \epsilon_{I}=\psi\circ \delta_{J}\circ \epsilon_{J}=0$. Since $M$ is torsion free ($t(M)=\mathrm{Ker}(\varphi_{M})=0$), we have that $\varphi_{M}$ is a monomorphism an then we get that $\beta\circ \epsilon_{I}=0$. Since $I(C,-)$,  $J(C,-)\in \mathcal{F}_{C}$ we get that $I(C,-)\cap J(C,-)\in \mathcal{F}_{C}$ (by property $T_{2}$) and therefore, by lemma 5.30 (ii) in \cite{Rotman},
we have that $w=\varphi_{I,M}^{C}(\beta)=0$ in $\mathbb{L}(M)$. We conclude that $t(\mathbb{L}(M))=0$.
\end{proof}

\begin{proposition}
$\mathbb{G}(M)$ is an $\mathcal{F}$-closed module for each $M\in \mathrm{Mod}(\mathcal{C})$.
\end{proposition}
\begin{proof}
$(1)$ Let us see that $$\xymatrix{\mathrm{Hom}_{\mathrm{Mod}(\mathcal{C})}\big(\mathrm{Hom}_{\mathcal{C}}(C,-),\mathbb{L}(\frac{M}{t(M)})\big)\ar[rr]_(.6){\theta_{J,\mathbb{L}(\frac{M}{t(M)})}^{C}} & &  \mathrm{Hom}_{\mathrm{Mod}(\mathcal{C})}\big(J,\mathbb{L}(\frac{M}{t(M)})\big),}$$
is injective for all $J=J(C,-)\in \mathcal{F}_{C}$.\\
Applying \ref{Lrestrictorsion} to $M/t(M)$, we have that $0=t(\mathbb{L}(\frac{M}{t(M)}))=\mathrm{Ker}(\varphi_{\mathbb{L}\big(\frac{M}{t(M)}\big)})$. Then we have an exact sequence
$$\xymatrix{0\ar[r] & \mathbb{L}(\frac{M}{t(M)})\ar[r]^{\varphi_{\mathbb{L}(\frac{M}{t(M)})}} & \mathbb{L}^{2}(\frac{M}{t(M)})\ar[r] & \mathrm{Coker}(\varphi_{\mathbb{L}(\frac{M}{t(M)})})\ar[r] & 0.}$$
By definition of $\varphi_{\mathbb{L}(\frac{M}{t(M)})}$ for each $C\in\mathcal{C}$ we have that
$[\varphi_{\mathbb{L}(\frac{M}{t(M)})}]_{C}:=[\psi_{\mathbb{L}(\frac{M}{t(M)})}]_{C}\circ Y_{C}$ where $Y_{C}$ is the Yoneda iso (see \ref{defimorfismovar}) and $[\psi_{\mathbb{L}(\frac{M}{t(M)})}]_{C}$ is such that the following commutes for all $J(C,-)\in\mathcal{F}_{C}$
$$\xymatrix{\mathrm{Hom}_{\mathrm{Mod}(\mathcal{C})}\big(\mathrm{Hom}_{\mathcal{C}}(C,-),\mathbb{L}(\frac{M}{t(M)})\big)\ar[rr]^{[\psi_{\mathbb{L}(\frac{M}{t(M)})}]_{C}}\ar@{=}[d] & & \varinjlim_{J\in \mathcal{F}_{C}}\mathrm{Hom}_{\mathrm{Mod}(\mathcal{C})}\big(J,\mathbb{L}(\frac{M}{t(M)})\big)\\
\mathrm{Hom}_{\mathrm{Mod}(\mathcal{C})}\big(\mathrm{Hom}_{\mathcal{C}}(C,-),\mathbb{L}(\frac{M}{t(M)})\big)\ar[rr]_(.6){\theta_{J,\mathbb{L}(\frac{M}{t(M)})}^{C}} & &  \mathrm{Hom}_{\mathrm{Mod}(\mathcal{C})}\big(J,\mathbb{L}(\frac{M}{t(M)})\big)\ar[u]_{\varphi_{J,\mathbb{L}(\frac{M}{t(M)})}^{C}}.}$$
Since $[\varphi_{\mathbb{L}(\frac{M}{t(M)})}]_{C}:=[\psi_{\mathbb{L}(\frac{M}{t(M)})}]_{C}\circ Y_{C}$ is mono and $Y_{C}$ is  iso, we conclude that $[\psi_{\mathbb{L}(\frac{M}{t(M)})}]_{C}$ is mono. By the above diagram we conclude that $\theta_{J,\mathbb{L}(\frac{M}{t(M)})}^{C}$ is mono for all $J_{C}\in \mathcal{F}_{C}$.\\
That is, we have that 

$$\xymatrix{\mathrm{Hom}_{\mathrm{Mod}(\mathcal{C})}\big(\mathrm{Hom}_{\mathcal{C}}(C,-),\mathbb{L}(\frac{M}{t(M)})\big)\ar[rr]_(.6){\theta_{J,\mathbb{L}(\frac{M}{t(M)})}^{C}} & &  \mathrm{Hom}_{\mathrm{Mod}(\mathcal{C})}\big(J,\mathbb{L}(\frac{M}{t(M)})\big),}$$
is injective for all $J=J(C,-)\in \mathcal{F}_{C}$.\\

$(2)$ Now, since $t(M/t(M))=0$ and $\mathrm{Ker}(\varphi_{\frac{M}{t(M)}})=t(M/t(M))=0$, we have the exact sequence
$$\xymatrix{0\ar[r] & \frac{M}{t(M)}\ar[r]^{\varphi_{\frac{M}{t(M)}}} & M_{\mathbb{L}}\ar[r] & \mathrm{Coker}(\varphi_{\frac{M}{t(M)}})	\ar[r] & 0}$$
Let $f:J(C,-)\longrightarrow \mathbb{G}(M)=\mathbb{L}(\frac{M}{t(M)})$ with $J(C,-)\in \mathcal{F}_{C}$. Taking the pullback, we have that
$$\xymatrix{0\ar[r] & I(C,-)\ar[r]^{\mu_{I,J}}\ar[d]^{g} & J(C,-)\ar[r]\ar[d]^{f}& 	\frac{J(C,-)}{I(C,-)}\ar[r]\ar[d]^{\theta} & 0\\
0\ar[r] & \frac{M}{t(M)}\ar[r]^{\varphi_{\frac{M}{t(M)}}} & \mathbb{L}(\frac{M}{t(M)})\ar[r] & \mathrm{Coker}(\varphi_{\frac{M}{t(M)}})	\ar[r] & 0}$$
where $\theta$ is a monomorphism.\\
We know that $\mathrm{Coker}(\varphi_{\frac{M}{t(M)}})$ is a torsion module, so we conclude that 
$\frac{J(C,-)}{I(C,-)}$ is a torsion module since $\mathcal{T}_{\mathcal{F}}$ is closed under subobjects (that is, $\frac{J(C,-)}{I(C,-)}\in \mathcal{T}_{\mathcal{F}}$). Then, we have that $\mathrm{Ann}(\overline{h},-)\in \mathcal{F}_{B}$ for every $\overline{h}\in \frac{J(C,B)}{I(C,B)}$ and for every $B\in \mathcal{C}$ (definition of $\mathcal{T}_{\mathcal{F}}$). But, $\mathrm{Ann}(\overline{h},-)=\Big(I(C,-):h\Big)$. That is, we have that $\Big(I(C,-):h\Big)\in \mathcal{F}_{B}$ for every $h\in J(C,B)$. By $T_{4}$, we conclude that
$I(C,-)\in \mathcal{F}_{C}$.
Considering the morphism $g:I(C,-)	\longrightarrow \frac{M}{t(M)}$ from the above diagram, we have an element $w:=\varphi_{I,M}^{C}(g)\in \mathbb{L}(\frac{M}{t(M)})(C)$. Then by  \ref{cuadradochido}, we have the diagram
$$\xymatrix{I(C,-)\ar[r]^{\delta_{I}}\ar[d]_{g} & \mathrm{Hom}_{\mathcal{C}}(C,-)\ar[d]^{\psi}\\
\frac{M}{t(M)}\ar[r]_{\varphi_{\frac{M}{t(M)}}} & \mathbb{L}(\frac{M}{t(M)})=\mathbb{G}(M)}$$ 
where $\psi$ corresponds to $w\in \mathbb{L}(\frac{M}{t(M)})(C)$ via Yoneda isomorphism.
By the above diagram que have that $\psi |_{I(C,-)}= \varphi_{\frac{M}{t(M)}}\circ g$. We assert that $\psi |_{J(C,-)}=\psi\circ \delta_{J}=f$. Indeed, we have that $f\circ \mu_{I,J}=\varphi_{\frac{M}{t(M)}}\circ g=\psi |_{I(C,-)}=\psi\circ \delta_{J}\circ \mu_{I,J}$. By \ref{igualrest}, we conclude that $\psi\circ \delta_{J}=f$. Proving that
$$\xymatrix{\mathrm{Hom}_{\mathrm{Mod}(\mathcal{C})}\big(\mathrm{Hom}_{\mathcal{C}}(C,-),\mathbb{L}(\frac{M}{t(M)})\big)\ar[rr]_(.6){\theta_{J,\mathbb{L}(\frac{M}{t(M)})}^{C}} & &  \mathrm{Hom}_{\mathrm{Mod}(\mathcal{C})}\big(J,\mathbb{L}(\frac{M}{t(M)})\big),}$$
is surjective for all $C\in \mathcal{C}$ and $J\in \mathcal{F}_{C}$, since $\theta_{J,\mathbb{L}(\frac{M}{t(M)})}^{C}(\psi)= \psi\circ \delta_{J}=f$.
\end{proof}

\begin{definition}
We denote by $\mathrm{Mod}(\mathcal{C},\mathcal{F})$ the full subcategory of $\mathrm{Mod}(\mathcal{C})$ consisting of $\mathcal{F}$-closed modules.
\end{definition}

\begin{proposition}
$\mathbb{G}:\mathrm{Mod}(\mathcal{C})\longrightarrow \mathrm{Mod}(\mathcal{C},\mathcal{F})$ is left adjoint to the inclusion $i:\mathrm{Mod}(\mathcal{C},\mathcal{F})\longrightarrow \mathrm{Mod}(\mathcal{C})$.
\end{proposition}
\begin{proof}
We have to show that there exists an isomorphism
$$\Phi_{M,N}:\mathrm{Hom}_{\mathrm{Mod}(\mathcal{C},\mathcal{F})}(\mathbb{G}(M),N)\longrightarrow \mathrm{Hom}_{\mathrm{Mod}(\mathcal{C})}(M,N)$$ for every $N\in \mathrm{Mod}(\mathcal{C},\mathcal{F})$ and $M\in \mathrm{Mod}(\mathcal{C})$.
We have $\varphi_{\frac{M}{t(M)}}:\frac{M}{t(M)}\longrightarrow \mathbb{L}(\frac{M}{t(M)})=\mathbb{G}(M)$ and $\pi_{M}:M\longrightarrow \frac{M}{t(M)}$ the projection. Given $\alpha:\mathbb{G}(M)\longrightarrow N$ we define $\Phi_{M,N}(\alpha):=\alpha\circ \varphi_{\frac{M}{t(M)}}\circ \pi_{M}=\alpha\circ \Delta_{M}$.\\ 
Now, given $f:M\longrightarrow N$ we have the following commutative diagram
$$\xymatrix{M\ar[r]^{f}\ar[d]^{\Delta_{M}} & N\ar[d]^{\Delta_{N}}\\
\mathbb{G}(M)\ar[r]^{\mathbb{G}(f)} & \mathbb{G}(N).}$$
Since $N\in \mathrm{Mod}(\mathcal{C},\mathcal{F})$, we have that $\Delta_{N}$ is an isomorphism (see \ref{Closeddeltaiso}). So we define
$\alpha:=\Delta_{N}^{-1}\circ \mathbb{G}$ and we have that
$\Phi_{M,N}(\alpha)=\alpha\circ \Delta_{M}= \Delta_{N}^{-1}\circ \mathbb{G}\circ \Delta_{M}=f$. So, we have that $\Phi_{M,N}$ is surjective.\\
Now, consider the exact sequence
$$\xymatrix{0\ar[r] & \frac{M}{t(M)}\ar[r]^{\varphi_{\frac{M}{t(M)}}} & \mathbb{L}(\frac{M}{t(M)})\ar[r] & \mathrm{Coker}(\varphi_{\frac{M}{t(M)}})=Z\ar[r] & 0}$$
Applying $\mathrm{Hom}_{\mathrm{Mod}(\mathcal{C})}(-,N)$ we have the exact sequence
$$\xymatrix{\mathrm{Hom}_{\mathrm{Mod}(\mathcal{C})}(Z,N)\ar[r] & \mathrm{Hom}_{\mathrm{Mod}(\mathcal{C})}(\mathbb{G}(M),N)\ar[r] & 
\mathrm{Hom}_{\mathrm{Mod}(\mathcal{C})}(\frac{M}{t(M)},N)}$$
Since $Z$ is torsion module and $N$ is a torsion-free module (see \ref{kervarphi} and \ref{Flibreclode}), we have that $\mathrm{Hom}_{\mathrm{Mod}(\mathcal{C})}(Z,N)=0$. Then we conclude that 
$$\xymatrix{\mathrm{Hom}_{\mathrm{Mod}(\mathcal{C})}(\mathbb{G}(M),N)\ar[rrr]^{\mathrm{Hom}_{\mathrm{Mod}(\mathcal{C})}(\varphi_{\frac{M}{t(M)}},N)} & & &
\mathrm{Hom}_{\mathrm{Mod}(\mathcal{C})}(\frac{M}{t(M)},N)}$$
is injective. Now, we assert that
$$\mathrm{Hom}_{\mathrm{Mod}(\mathcal{C})}\Big(\pi_{M},N\Big):\mathrm{Hom}_{\mathrm{Mod}(\mathcal{C})}\Big(\frac{M}{t(M)},N\Big)\longrightarrow \mathrm{Hom}_{\mathrm{Mod}(\mathcal{C})}\Big(M,N\Big)$$ is an isomorphism.
This follows from the following exact sequence
$$\xymatrix{0\ar[r] & \Big(\frac{M}{t(M)},N\Big)\ar[r] &\Big(M,N\Big)\ar[r] & \Big(t(M),N\Big)=0}$$
where $\Big(t(M),N\Big)=0$ because $t(M)$ is torsion and $N$ is torsion-free. 
This implies that 
$$\xymatrix{\mathrm{Hom}_{\mathrm{Mod}(\mathcal{C},\mathcal{F})}(\mathbb{G}(M),N)\ar[rrr]^{\mathrm{Hom}_{\mathrm{Mod}(\mathcal{C})}(\Delta_{M},N)} & & &
\mathrm{Hom}_{\mathrm{Mod}(\mathcal{C})}(M,N)}$$
is injective. But $\Phi_{M,N}=\mathrm{Hom}_{\mathrm{Mod}(\mathcal{C})}(\Delta_{M},N)$. Proving that $\Phi_{M,N}$ is bijective.
\end{proof}

We  recall the following notion. Let $\mathcal{X}\subseteq \mathrm{Mod}(\mathcal{C})$ a class of objects and $f:M\longrightarrow X$ a morphism with $X\in \mathcal{X}$, it is said that $f$ is a left $\mathcal{X}$-approximation of $M$ if for every morphism $g:M\longrightarrow X'$ with $X'\in \mathcal{X}$ there exists a morphism $h:X\longrightarrow X'$ such that $g=h\circ f$. If every object in $\mathrm{Mod}(\mathcal{C})$ admits a left $\mathcal{X}$-approximation we say that $\mathcal{X}$ is covariantly finite in $\mathrm{Mod}(\mathcal{C})$.

\begin{corollary}
The morphism $\Delta_{M}:M\longrightarrow \mathbb{G}(M)$ is a left $\mathrm{Mod}(\mathcal{C},\mathcal{F})$-approximation of $M$. In particular $\mathrm{Mod}(\mathcal{C},\mathcal{F})$ is a covariantly finite subcategory of $\mathrm{Mod}(\mathcal{C})$.
\end{corollary}
\begin{proof}
This follows by the previous propositions, since $\mathbb{G}(M)$ is closed.
\end{proof}

\begin{corollary}\label{Gexact}
The functor $\mathbb{G}$ is exact.
\end{corollary}
\begin{proof}
The functor $\mathbb{G}$ is left exact since $\mathbb{L}$ is left exact and by \ref{otherdescLL}(b). Now $\mathbb{G}$ is left exact since $\mathbb{G}$ is left adjoint to $i$, see \cite[Theorem 2.6.1]{Weibel}.
\end{proof}

\begin{proposition}\label{KerG}
Let $\mathcal{F}=\{\mathcal{F}_{C}\}_{C\in \mathcal{C}}$ be a left Gabriel filter and $M\in \mathrm{Mod}(\mathcal{C})$. Then $\mathbb{G}(M)=0$ if and only if $M$ is an $\mathcal{F}$-torsion module. That is  $\mathbb{G}(M)=0$ if and only if $t(M)=M$, where $t$ is the radical associated to $\mathcal{F}$.
\end{proposition}
\begin{proof}
Let $t:\mathrm{Mod}(\mathcal{C})\longrightarrow \mathrm{Mod}(\mathcal{C})$ the radical associated to $\mathcal{F}$. 
If $M$ is $\mathcal{F}$-torsion module  we have that $t(M)=M$. Then
$\mathbb{G}(M)=\mathbb{L}(\frac{M}{t(M)})=0$.\\
Now, if $\mathbb{G}(M)=\mathbb{L}(\frac{M}{t(M)})=0$ by \ref{L=0siMtor} we have that $\frac{M}{t(M)}$ is $\mathcal{F}$-torsion module. That is $t(\frac{M}{t(M)})=\frac{M}{t(M)}$ but since $t$ is radical we have that $t(\frac{M}{t(M)})=0$. Then we have that $t(M)=M$ and therefore $M$ is an $\mathcal{F}$-torsion module. 
\end{proof}

\begin{definition}
Let $\mathcal{A}$ be a complete Grothendieck category an let $\mathcal{B}$ a full subcategory of $\mathcal{A}$. 
\begin{enumerate}
\item [(a)] $\mathcal{B}$ is reflective subcategory of $\mathcal{A}$ if the inclusion functor $i:\mathcal{B}\longrightarrow \mathcal{A}$ has a left adjoint $a:\mathcal{A}\longrightarrow \mathcal{B}$.

\item [(b)] $\mathcal{B}$ is a Giraud subcategory of $\mathcal{A}$ if $\mathcal{B}$ is reflective such that the funtor  $a:\mathcal{A}\longrightarrow \mathcal{B}$ preserve kerneles.
\end{enumerate}
\end{definition}
It is well known that if $\mathcal{B}$ a Giraud subcategory of $\mathcal{A}$, then the functor $a:\mathcal{A}\longrightarrow \mathcal{B}$ is exact an $\mathcal{B}$ is a Grothendieck category.\\
Given the adjoint pair $(a,i)$, we get the unit and counit morphisms
$$\eta: 1_{\mathcal{A}}\longrightarrow i\circ a,\quad \epsilon: a\circ i\longrightarrow 1_{\mathcal{B}}.$$
We have the following well known result. For convenience of the reader we include a proof.
\begin{proposition}\label{giarutorsion}
Let $\mathcal{A}$ be a complete Grothendieck category, $\mathcal{B}$ a Giraud subcategory of $\mathcal{A}$ and $a:\mathcal{A}\longrightarrow \mathcal{B}$ the left adjoint to the inclusion functor. We set $\mathcal{T}_{\mathcal{B}}:=\{A\in \mathcal{A}\mid a(A)=0\}$ and $\mathcal{F}_{\mathcal{B}}:=\{A\in \mathcal{A}\mid \eta_{A}:A\longrightarrow i\circ a(A)\,\,\text{is monomorphism}\}.$
Then $(\mathcal{T}_{\mathcal{B}},\mathcal{F}_{\mathcal{B}})$ is a hereditary torsion theory for $\mathcal{A}$.
\end{proposition}
\begin{proof}
From the exactness of $a$ follows that $\mathcal{T}_{\mathcal{B}}$ is closed under subobjects quotients and extensions. Now, since $a$ is left adjoint to the inclsuion $i$ we have that $\mathcal{T}_{\mathcal{B}}$ is closed under coproducts. Therefore $\mathcal{T}_{\mathcal{B}}$ is a hereditary torsion class. Now, let $\mathcal{X}$ be the torsion free class associated to $\mathcal{T}_{\mathcal{B}}$. Let us see that $\mathcal{X}=\mathcal{F}_{\mathcal{B}}$. Let $X\in \mathcal{X}$, then
$\mathrm{Hom}_{\mathcal{A}}(T,X)=0$ for all $T\in \mathcal{T}_{\mathcal{B}}$. For $X$ the unit of the adjunction, give us the exact sequence
$$\xymatrix{0\ar[r]  & K\ar[r]^{j} & X\ar[r]^{\eta_{X}} & (i\circ a)(X)}$$
Since $a$ is exact we have the following exact sequence
$$\xymatrix{0\ar[r]  & a(K)\ar[r]^{j} & a(X)\ar[r]^(.4){a(\eta_{X})} & a((i\circ a)(X))}$$
Now, by the triangular identities (see \cite[Theorem 3.1.5]{Borceux1}) we have that $\epsilon_{a(X)}\circ a(\eta_{X})=1_{a(X)}$.  Then $a(\eta_{X})$ is a monomorfismo and thus we have that $a(K)=0$. Therefore we have that $K\in \mathcal{T}_{\mathcal{E}}$. Then we have that $\mathrm{Hom}_{\mathcal{A}}(K,X)=0$ and we conclude that $\eta_{X}$ is a monomorphism, proving that $X\in \mathcal{F}_{\mathcal{B}}$. Conversely, first for $X\in \mathcal{B}$ and $T\in \mathcal{T}_{\mathcal{B}}$ we have that
$0=\mathrm{Hom}_{\mathcal{A}}(a(T), X)=\mathrm{Hom}_{\mathcal{A}}(T,X)$. Now, suppose that $X\in \mathcal{F}_{\mathcal{B}}$ and let 
$\alpha:T\longrightarrow X$. Then we have $\eta_{X}\circ \alpha=0$ because $a(X)\in \mathcal{B}$, and since since $\eta_{X}$ is a monomorphism we have that $\alpha=0$. Thus, we have proved that
$\mathrm{Hom}_{\mathcal{A}}(T,X)=0$  for all $T\in \mathcal{T}_{F}$. Proving that $\mathcal{F}_{\mathcal{B}}=\mathcal{X}$.
\end{proof}

\begin{proposition}
The category $\mathrm{Mod}(\mathcal{C},\mathcal{F})$ is a Giraud subcategory of $\mathrm{Mod}(\mathcal{C})$
\end{proposition}
\begin{proof}
This follows from \ref{Gexact}.
\end{proof}

\begin{corollary}
$\mathrm{Mod}(\mathcal{C},\mathcal{F})$ is a Grothendiek abelian category.
\end{corollary}
We have the main result of this paper which is a generelization of a classical result given by P. Gabriel in his doctoral thesis (see \cite[Lemme 1]{Gabriel1} on page 41)

\begin{theorem}\label{Biyeccionbuena}
There exists a bijective correspondence between Gabriel filters on $\mathcal{C}$ and the class of isomorphisms of Giraud subcategories of $\mathrm{Mod}(\mathcal{C})$.
\end{theorem}
\begin{proof}
By \ref{giarutorsion}, there is a map
$$\mathbb{T}:\{\text{Giraud subcategories of}\,\,\mathrm{Mod}(\mathcal{C})\}\longrightarrow \{\text{hereditary torsion classes of}\,\,\mathrm{Mod}(\mathcal{C})\},$$
given by $\mathbb{T}(\mathcal{B}):=\mathcal{T}_{\mathcal{B}}=\left\{M\in \mathrm{Mod}(\mathcal{C})\mid a'(M)=0\right\},$
where $a':\mathrm{Mod}(\mathcal{C})\longrightarrow \mathcal{B}$ is the left adjoint to the inclusion $i':\mathcal{B}\longrightarrow \mathrm{Mod}(\mathcal{C})$. By \ref{biyeGafilterpretor}, there is a biyective correspondence
$$\xymatrix{\{\text{Gabriel filters of}\,\,\mathcal{C}\}
\ar@<1ex>[d]^{\Psi}\\
\{\text{hereditary torsion classes of}\,\,\mathrm{Mod}(\mathcal{C})\}
\ar@<1ex>[u]^{\Psi^{-1}}.}$$
We define a correspondence
$$\Phi:\{\text{Giraud subcategories of}\,\,\mathrm{Mod}(\mathcal{C})\}\longrightarrow \{\text{Gabriel filters of}\,\,\mathcal{C}\},$$
as $\Phi:=\Psi^{-1}\circ \mathbb{T}.$\\
Then, for $\mathcal{B}$ a Giraud subcategory of $\mathrm{Mod}(\mathcal{C})$  we have  that
$\Phi(\mathcal{B}):=\mathcal{F}=\{\mathcal{F}_{C}\}_{C\in \mathcal{C}}$ where 
\begin{align*}
 \mathcal{F}_{C} & :=\left\{I\subseteq \mathrm{Hom}_{\mathcal{C}}(C,-)\mid \frac{\mathrm{Hom}_{\mathcal{C}}(C,-)}{I}\in \mathcal{T}_{\mathcal{B}}\right\}\\
 & =\left\{I\subseteq \mathrm{Hom}_{\mathcal{C}}(C,-)\mid a'\Big( \frac{\mathrm{Hom}_{\mathcal{C}}(C,-)}{I}\Big)=0\right\}.
\end{align*}
We define 
$$\Gamma:\{\text{Gabriel filters of}\,\,\mathcal{C}\}\longrightarrow \{\text{Giraud subcategories of}\,\,\mathrm{Mod}(\mathcal{C})\},$$
as $\Gamma(\mathcal{F})=\mathrm{Mod}(\mathcal{C},\mathcal{F})$.\\
Lets see that $\Phi\circ \Gamma=1$. Indeed, let $\mathcal{F}=\{\mathcal{F}_{C}\}_{C\in \mathcal{C}}$ a Gabriel filter. By \ref{biyeGafilterpretor}, $\mathcal{F}$ corresponds biyectively to
$$\mathcal{T}_{\mathcal{F}}:=\left\{M\in \mathrm{Mod}(\mathcal{C})\mid  \text{for each}\,\, C\in \mathcal{C},\,\,\, \mathrm{Ann}(x,-)\in \mathcal{F}_{C}\,\,\forall x\in M(C) \right\}$$
By \ref{remtorsiorad}, we have that $\mathcal{T}_{\mathcal{F}}=\{M\in \mathrm{Mod}(\mathcal{C})\mid t(M)=M\}$ where $t$ is the radical associated to $\mathcal{F}$.\\
On the other hand, by construction of $\Phi$ we have that Gabriel filter
$\Phi(\Gamma(\mathcal{F}))=\Phi(\mathrm{Mod}(\mathcal{C},\mathcal{F}))$ corresponds biyectively to the torsion class
$\{M\in \mathrm{Mod}(\mathcal{C})\mid \mathbb{G}(M)=0\}$. By \ref{KerG} we have that $\mathcal{T}_{\mathcal{F}}=\{M\in \mathrm{Mod}(\mathcal{C})\mid \mathbb{G}(M)=0\}$. Then we have that $\Phi(\Gamma(\mathcal{F}))=\mathcal{F}$, proving that $\Phi\circ\Gamma=1$.\\

Let $\mathcal{B}$ be a Giraud subcategory of $\mathrm{Mod}(\mathcal{C})$ and $a':\mathrm{Mod}(\mathcal{C})\longrightarrow \mathcal{B}$ the left adjoint to the inclusion $i':\mathcal{B}\longrightarrow \mathrm{Mod}(\mathcal{C})$. Then we have the torsion class $\mathbb{T}(\mathcal{B}):=\mathcal{T}_{\mathcal{B}}=\left\{M\in \mathrm{Mod}(\mathcal{C})\mid a'(M)=0\right\}.$\\
By the definition $\Phi$ and the bijective correspondences  \ref{biyeGafilterpretor} and \ref{biyepretorprerad} we have that
the torsion class $\mathcal{T}_{\mathcal{B}}$ determine a unique filter Gabriel $\Phi(\mathcal{T}_{\mathcal{B}})=\mathcal{F}=\{\mathcal{F}_{C}\}_{C\in \mathcal{C}}$ where
\begin{align*}
 \mathcal{F}_{C}=\left\{I\subseteq \mathrm{Hom}_{\mathcal{C}}(C,-)\mid a'\Big( \frac{\mathrm{Hom}_{\mathcal{C}}(C,-)}{I}\Big)=0\right\}
\end{align*}
and a unique radical $t:\mathrm{Mod}(\mathcal{C})\longrightarrow \mathrm{Mod}(\mathcal{C})$. That is we have the assignation
$$(\ast):\quad \Phi(\mathcal{T}_{\mathcal{B}})=\mathcal{F}=\{ \mathcal{F}_{C}\} \longleftrightarrow
\mathcal{T}_{\mathcal{B}} \longleftrightarrow t.$$
Now, we can construct $\mathrm{Mod}(\mathcal{C},\mathcal{F})$ and $\mathbb{G}:\mathrm{Mod}(\mathcal{C})\longrightarrow \mathrm{Mod}(\mathcal{C},\mathcal{F})$. By \ref{KerG}, we have that $\mathbb{G}(M)=0$ if and only if $t(M)=M$. But by the correspondence \ref{biyepretorprerad}, we have that $t(M)=M$ if and only if $M\in \mathcal{T}_{\mathcal{B}}$ (the torsion class associated to $t$ is  $\mathcal{T}_{\mathcal{B}}$ by ($\ast$)). Then we have that 
$$(\star):\quad\mathcal{T}_{\mathcal{B}}=\left\{M\in \mathrm{Mod}(\mathcal{C})\mid a'(M)=0\right\}=\left\{M\in \mathrm{Mod}(\mathcal{C})\mid \mathbb{G}(M)=0\right\}.$$
Now, we will show that $\mathcal{B}$ is equivalent to $\mathrm{Mod}(\mathcal{C},\mathcal{F})$. Indeed, consider the following commutative diagrama
$$\xymatrix{\mathcal{B}\ar@<1ex>[dd]^(.45){\mathbb{G}\circ i'}\ar@<1ex>[drr]^(.45){i'}  &\\
& & \mathrm{Mod}(\mathcal{C})\ar@<1ex>[ull]^(.45){a'}\ar@<1ex>[dll]^(.45){\mathbb{G}}\\
\mathrm{Mod}(\mathcal{C},\mathcal{F})\ar@<1ex>[uu]^(.45){a'\circ i}\ar@<1ex>[urr]^(.45){i}}$$
We assert that $a'\circ i\circ \mathbb{G} \simeq a'$. Indeed, by \ref{kerDelta} we have exact sequence
$$\gamma:\quad \xymatrix{0\ar[r] & t(M)\ar[r] & M\ar[r]^{\Delta_{M}} & \mathbb{G}(M)\ar[r] & \mathrm{Coker}(\varphi_{\frac{M}{t(M)}})\ar[r] & 0}$$
In particular, $\mathrm{Ker}(\Delta_{M})$ and $\mathrm{Coker}(\Delta_{M})$ are of $\mathcal{F}$-torsion (i.e, $\mathrm{Ker}(\Delta_{M}),\mathrm{Coker}(\Delta_{M}) \in \mathcal{T}_{\mathcal{B}}$). Since $\mathcal{T}_{\mathcal{B}}=\{M\in \mathrm{Mod}(\mathcal{C})\mid a'(M)=0\}$ (see equality $(\star)$), we have that applying $a'$ to the exact sequence $\gamma$ we have an isomorphism (recall that $a'$ is exact)
$$\xymatrix{a'(M)\ar[r]^{a'(\Delta_{M})} & a'(\mathbb{G}(M))}$$
It is easy to show that this isomorphism is natural and then we have  that $a'\circ i\circ \mathbb{G} \simeq a'$. Now, consider the unit $\eta':1_{\mathrm{Mod}(\mathcal{C})}\longrightarrow i'\circ a'.$ Since $i'$ is full and faithfull by \cite[Proposition 3.4.1]{Borceux1} on page 114, we  have that the counit $\epsilon':a'\circ i'\longrightarrow 1_{\mathcal{B}}$ is an isomorphism and we have that $a'\ast\eta'$ is isomorphism. That is, for all $M\in \mathrm{Mod}(\mathcal{C})$ we get that $a'(\eta'_{M})$ is an isomorphism.
Consider the exact sequence
$$(\delta):\quad \xymatrix{0\ar[r] & \mathrm{Ker}(\eta'_{M})\ar[r] & M\ar[r]^{\eta'_{M}}\ar[r] & i'a'(M)\ar[r] & \mathrm{Coker}(\eta'_{M})\ar[r] & 0}$$
Since $a'$ is exact and  $a'(\eta'_{M})$ is an isomorphism, we have that $a'(\mathrm{Ker}(\eta'_{M}))=0$ and $a'(\mathrm{Coker}(\eta'_{M}))=0$. This implies that $\mathrm{Ker}(\eta'_{M}),\mathrm{Coker}(\eta'_{M})\in \mathcal{T}_{\mathcal{B}}$. Then  $\mathrm{Ker}(\eta'_{M}),\mathrm{Coker}(\eta'_{M})$ are of $\mathcal{F}$-torsion. Aplying  $\mathbb{G}$ to the exact sequence $(\delta)$, we have the isomorphism $\mathbb{G}(\eta'_{M}):\mathbb{G}(M)\longrightarrow \mathbb{G}\circ i'\circ a'(M)$. Then we have that
$$\mathbb{G}\simeq \mathbb{G}\circ i'\circ a'.$$
Since the counit $\epsilon:\mathbb{G}\circ i\longrightarrow \mathrm{Mod}(\mathcal{C},\mathcal{F})$ is an isomorphism (by the same reason for $\epsilon'$). We conclude that
$(\mathbb{G}\circ i')\circ (a'\circ i)=(\mathbb{G}\circ i'\circ a')\circ i\simeq \mathbb{G}\circ i\simeq 1$ where the last isomorphism is via $\epsilon$. Similarly, $(a'\circ i)\circ (\mathbb{G}\circ i')=(a'\circ i\circ \mathbb{G})\circ i'\simeq a'\circ i'\circ 1$. Proving that  $\mathcal{B}$ is equivalent to $\mathrm{Mod}(\mathcal{C},\mathcal{F})$. Then $\Gamma\circ \Phi=1$.
\end{proof}

The notion of quotient and localization of abelian categories by dense subcategories (i.e., Serre classes) was introduced by P. Gabriel in his famous  Doctoral thesis ``Des cat\'egories ab\'elienne'' \cite{Gabriel1}, and plays an important role in ring theory. This notion achieves some goal as quotients in other area of mathematics.\\
Let $\mathcal{A}$ be an abelian category. Recall that a Serre subcategory $\mathcal{B}$ of $\mathcal{A}$ is a subcategory closed under forming subobjects, quotients and extensions. In this case we can construct the quotient category $\mathcal{A}/\mathcal{B}$ of $\mathcal{A}$ with respect to $\mathcal{B}$ and a functor $Q:\mathcal{A}\longrightarrow \mathcal{A}/\mathcal{B}$ which is called the quotient functor. For basic proporties of quotient categories we refer to \cite{Popescu} and \cite{Gabriel1}. We recall the following well known result

\begin{theorem}[P. Gabriel]
Let $\mathcal{A}$ be a locally small abelian category and $\mathcal{B}$ a Serre subcategory. Consider $\Sigma=\{f\in \mathcal{A}\mid \mathrm{Ker}(f),\mathrm{Coker}(f)\in \mathcal{B}\}$.  Then there exists an abelian category $\mathcal{A}/\mathcal{B}$ and a exact functor $Q:\mathcal{A}\longrightarrow \mathcal{A}/\mathcal{B}$ such that $Q(f)$ is an isomorphism for all $f\in \Sigma$ and if $F:\mathcal{A}\longrightarrow \mathcal{D}$ is a functor satisfying that $F(f)$ is an isomorphism for all $f\in \Sigma$, then there exists a unique functor $G:\mathcal{A}/\mathcal{B}\longrightarrow \mathcal{D}$ such that $F=G\circ Q$.
\end{theorem}

\begin{proposition}\label{qutirn}
There exists an equivalence of categories
$$\mathrm{Mod}(\mathcal{C},\mathcal{F})\simeq \mathrm{Mod}(\mathcal{C})/\mathcal{T}_{\mathcal{F}}.$$
\end{proposition}
\begin{proof}
We have an exact functor $\mathbb{G}:\mathrm{Mod}(\mathcal{C})\longrightarrow 	\mathrm{Mod}(\mathcal{C},\mathcal{F})$ whose right adjoint $i:\mathrm{Mod}(\mathcal{C},\mathcal{F})\longrightarrow \mathrm{Mod}(\mathcal{C})$ is full and faithfull. By \cite[Theorem 4.9]{Popescu} in page 180, we have that $\mathrm{Ker}(\mathbb{G})$ is a localizing subcategory and 
$$\mathrm{Mod}(\mathcal{C},\mathcal{F})\simeq \mathrm{Mod}(\mathcal{C})/\mathrm{Ker}(\mathbb{G}).$$
By \ref{KerG}, we conclude that $\mathrm{Ker}(\mathbb{G})=\mathcal{T}_{\mathcal{F}}$.
\end{proof}
Finally we give an example of a left Gabriel filter using the path category of a quiver.
\begin{example}\label{example}
Consider a field $K$ and the quiver
$$Q:\quad \xymatrix{\cdots  & i-1\ar@/^1pc/[r]^{\alpha_{i-1}}  & i\ar@/^1pc/[l]^{\beta_{i-1}}\ar@/^1pc/[r]^{\alpha_{i}} & i+1\ar@/^1pc/[l]^{\beta_{i}}\ar@/^1pc/[r]^{\alpha_{i+1}} & i+1\ar@/^1pc/[l]^{\beta_{i+1}} & \cdots}.$$ and $\mathcal{C}=KQ/\langle \rho\rangle$ the path category with the relations given by $\alpha_{i}\circ\alpha_{i-1}=\beta_{i}\circ \beta_{i+1}=\alpha_{i}\circ \beta_{i}=0$. Then the functor $(i,-):=\mathrm{Hom}_{KQ/\langle \rho\rangle}(i,-)$ can be thought as a representation in the category $\mathrm{Rep}(Q,\rho)$  given by 
$$\Big( (i,j), u:i\rightarrow j\Big)_{j\in Q_{0},u\in Q_{1}}.$$
Thus $(i,i):=\{f:i\longrightarrow i\}=\langle 1_{i},\beta_{i}\alpha_{i}\rangle\simeq K^{2}$, $(i,i+1)=\{f:i\longrightarrow i+1\}=\langle \alpha_{i}\rangle \simeq K$ and $(i,i-1)=\{f:i\longrightarrow i-1\}=\langle \beta_{i-1} \rangle\simeq K$. Then, the representation corresponding to $(i,-)$ is the following
$$\xymatrix{\cdots  & K\ar@/^1pc/[r]^{0}  & K^{2}\ar@/^1pc/[l]^{(1,0)}\ar@/^1pc/[r]^{(1,0)} & K\ar@/^1pc/[l]^{{\left( \begin{smallmatrix}0\\ 
1
\end{smallmatrix}\right)}}  & \cdots}.$$  Computing we have that $(i,-)$ has 7 ideals: $0$, $(i,-)$ and the given by the following representations

$(i,-)$ is the following
$$[\beta_{i}\alpha_{i}]:\xymatrix{\cdots  & 0\ar@/^1pc/[r]^{0}  & K\ar@/^1pc/[l]^{0}\ar@/^1pc/[r]^{0} & 0\ar@/^1pc/[l]^{0}  & \cdots}.$$ 

$$[\beta_{i-1}]:\xymatrix{\cdots  & K\ar@/^1pc/[r]^{0}  & 0\ar@/^1pc/[l]^{0}\ar@/^1pc/[r]^{0} & 0\ar@/^1pc/[l]^{0}  & \cdots}.$$ 

$$[\alpha_{i}]:\xymatrix{\cdots  & 0\ar@/^1pc/[r]^{0}  & K\ar@/^1pc/[l]^{0}\ar@/^1pc/[r]^{0} & K\ar@/^1pc/[l]^{1}  & \cdots}.$$ 

$$[\beta_{i-1}]\oplus [\alpha_{i}\beta_{i}]:\xymatrix{\cdots  & K\ar@/^1pc/[r]^{0}  & K\ar@/^1pc/[l]^{0}\ar@/^1pc/[r]^{0} & 0\ar@/^1pc/[l]^{0}  & \cdots}.$$ 

$$[\beta_{i-1}]\oplus [\alpha_{i}]:\xymatrix{\cdots  & K\ar@/^1pc/[r]^{0}  & K\ar@/^1pc/[l]^{0}\ar@/^1pc/[r]^{0} & K\ar@/^1pc/[l]^{1}  & \cdots}.$$ 
We get the following lattice of ideals
$$
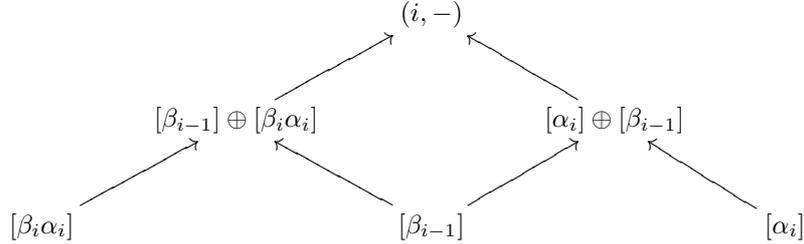
\xymatrix{  & & (i,-)\\
& [\beta_{i-1}]\oplus [\beta_{i}\alpha_{i}]\ar[ur] & & [\alpha_{i}]\oplus[\beta_{i-1}]\ar[ul]\\
[\beta_{i}\alpha_{i}]\ar[ur] & &  [\beta_{i-1}]\ar[ul]\ar[ur] & &  [\alpha_{i}]\ar[ul]}$$
Then an easy calculation shows that $\mathcal{F}=\{[\beta_{i-1}],[\alpha_{i}],[\alpha_{i}]\oplus [\beta_{i-1}], (i,-)\}$ is a left Gabriel filter.
It is easy to show in this case that $(i,-)\in \mathcal{T}_{\mathcal{F}}$ so we have that $\mathbb{G}\Big((i,-)\Big)=0$.
\end{example}

\section{Acknowledgements}
This investigation was iniciated after a series of talks that the first and the third authors gave at Universidad Aut\'onoma Metropolitana campus Iztapalapa, invited by the second author. The authors thanks project:{\it Apoyo a la Incorporaci\'on de Nuevos PTC-2019 PRODEP}; Grant ID: F-PROMEP-39/Rev-04
   SEP-23-005.


\bibliographystyle{alphadin}
\bibliography{GLFC}

\footnotesize

\vskip3mm \noindent Martin Ort\'iz Morales:\\ Facultad de Ciencias, Universidad  Aut\'onoma del Estado de M\'exico\\
T\'oluca, M\'exico.\\
{\tt mortizmo@uaemex.mx}

\vskip3mm \noindent Martha Lizbeth Shaid Sandoval Miranda:\\ Departamento de Matem\'aticas, Universidad Aut\'onoma Metropolitana Unidad Iztapalapa\\
Av. San Rafael Atlixco 186, Col. Vicentina Iztapalapa 09340, M\'exico, Ciudad de M\'exico.\\ {\tt marlisha@xanum.uam.mx, marlisha@ciencias.unam.mx}

\vskip3mm \noindent Valente Santiago Vargas:\\ Departamento de Matem\'aticas, Facultad de Ciencias, Universidad Nacional Aut\'onoma de M\'exico\\
Circuito Exterior, Ciudad Universitaria,
C.P. 04510, M\'exico, Ciudad de M\'exico.\\ {\tt valente.santiago.v@gmail.com.mx}

\end{document}